\documentclass{amsart}
\title{Stable Grothendieck Rings of Wreath Product Categories}
\author{Christopher Ryba}
\address{Department of Mathematics, Massachusetts Institute of Technology, Cambridge, MA 02139, USA}
\email{ryba@mit.edu}
\date{\today}
\usepackage{hyperref}
\usepackage{fullpage}
\usepackage{amsmath}
\usepackage{amsthm}
\usepackage{amsfonts}
\usepackage{amssymb}
\usepackage{ytableau}
\usepackage{subfig}
\usepackage{graphicx, caption, float}
\usepackage{verbatim}
\usepackage{todonotes}
\usepackage{ytableau}
\usepackage{esvect}
\begin{document}
\maketitle

\newtheorem{theorem}{Theorem}[section]
\newtheorem{lemma}[theorem]{Lemma}
\newtheorem{proposition}[theorem]{Proposition}
\newtheorem{corollary}[theorem]{Corollary}
\newtheorem{definition}[theorem]{Definition}
\newtheorem{remark}[theorem]{Remark}
\newtheorem{example}[theorem]{Example}

\newcommand\numberthis{\addtocounter{equation}{1}\tag{\theequation}}

\newcommand{\Res}[2] {
  \textnormal{Res}_{#1}^{#2} 
}

\newcommand{\Ind}[2] {
  \textnormal{Ind}_{#1}^{#2} 
}

\newcommand{\Coind}[2] {
  \textnormal{Coind}_{#1}^{#2} 
}

\newcommand{\Wreath}[1] {
  \mathcal{W}_{#1}(\mathcal{C})
}

\newcommand{\Groth}[0] {
  \mathcal{G}(\mathcal{C})
}

\newcommand{\End}[0] {
  \textnormal{End}
}

\newcommand{\Mat}[0] {
  \textnormal{Mat}
}

\newcommand{\Id}[0] {
  \textnormal{Id}
}

\newcommand{\hopf}[0] {
  \mathcal{R}
}

\begin{abstract}
Let $k$ be an algebraically closed field of characteristic zero, and let $\mathcal{C} = \hopf-mod$ be the category of finite-dimensional modules over a fixed Hopf algebra over $k$. One may form the wreath product categories $\Wreath{n} = (\hopf \wr S_n)-mod$ whose Grothendieck groups inherit the structure of a ring. Fixing distinguished generating sets (called basic hooks) of the Grothendieck rings, the classification of the simple objects in $\Wreath{n}$ allows one to demonstrate stability of structure constants in the Grothendieck rings (appropriately understood), and hence define a limiting Grothendieck ring. This ring is the Grothendieck ring of the wreath product Deligne category $S_t(\mathcal{C})$. We give a presentation of the ring and an expression for the distinguished basis arising from simple objects in the wreath product categories as polynomials in basic hooks. We discuss some applications when $\hopf$ is the group algebra of a finite group, and some results about stable Kronecker coefficients. Finally, we explain how to generalise to the setting where $\mathcal{C}$ is a tensor category.
\end{abstract}

\tableofcontents

\section{Introduction}
\noindent
In this paper, we consider the Grothendieck groups of the categories of modules over the wreath products $\hopf \wr S_n$, where $\hopf$ is a Hopf algebra, and $S_n$ is the symmetric group on $n$ symbols. The Hopf algebra structure gives rise to a multiplication on the Grothendieck groups, so we may speak of Grothendieck rings. We study the ring structure in the ``limit'' $n \to \infty$.
\newline \newline \noindent
Using Mackey theory, we describe the multiplication on the Grothendieck rings in terms of data associated to $\hopf$ and $S_n$. We show that this multiplication exhibits a certain stability property which allow us to define a ``limiting Grothendieck ring'', $\mathcal{G}_\infty(\mathcal{C})$ (here $\mathcal{C} = \hopf-mod$). This ring is the Grothendieck ring of the wreath product Deligne categories $S_t(\mathcal{C})$ introduced in \cite{Mori} and considered in \cite{Nate}. When $\mathcal{C}$ is the category of finite-dimensional vector spaces over the field $k$ (characteristic zero and algebraically closed), we recover the original Deligne category $\underline{\mbox{Rep}}(S_t)$. 
\newline \newline \noindent
Our first main result is Theorem \ref{strThm}, which gives the structure of $\mathcal{G}_\infty(\mathcal{C})$. If $\mathcal{G}(\mathcal{C})_i$ denotes a copy of the Grothendieck group of $\mathcal{C}$ with rational coefficients, we have a Lie algebra structure coming from the associative multiplication. This allows us to take the universal enveloping algebra $U(\mathcal{G}(\mathcal{C})_i)$. Then:
\[
\mathcal{G}_{\infty}(\mathcal{C}) = \bigotimes_{i=1}^{\infty} U(\mathcal{G}(C)_i)
\]
This generalises the fact (due to Deligne) that the Grothendieck ring of $\underline{\mbox{Rep}}(S_t)$ is the free polynomial algebra on certain elements (see \cite{deligne}); in \cite{Nate} these elements were generalised to \emph{basic hooks}. It was proved in \cite{Nate} that the basic hooks generate the Grothendieck ring of $S_t(\mathcal{C})$, although for arbitrary $\mathcal{C}$ they do not commute (so the Grothendieck ring is not a free polynomial algebra).
\newline \newline \noindent
Our second main result is about the Grothendieck ring of wreath product Deligne categories. Knowing that $\mathcal{G}_\infty(\mathcal{C})$ is isomorphic to the Grothendieck ring of $S_t(\mathcal{C})$, there is a natural basis of $\mathcal{G}_\infty(\mathcal{C})$, $X_{\vv{\lambda}}$, coming from the images of the indecomposable objects of $S_t(\mathcal{C})$. In Theorem \ref{bigthm} we give a generating function that describes the $X_{\vv{\lambda}}$ basis of $\mathcal{G}_\infty(\mathcal{C})$ in the presentation in Theorem \ref{strThm}. Finally, we discuss implications of these results for the asymptotic representation theory of wreath products and symmetric groups, and explain that all our results actually hold when $\mathcal{C}$ is a tensor category, without the need for a Hopf algebra $\hopf$.
\newline \newline \noindent
The outline of the paper is as follows. In Section 3 we establish notation relating to symmetric functions,  then in Section 4 we discuss wreath products and their irreducible representations. In Section 5 we apply Mackey theory to show tensor products of irreducible representations of wreath products decompose in ways controlled by double coset representatives of Young subgroups of symmetric groups, which we discuss in Section 6. We see that the double coset representatives exhibit certain stability properties that allow us to define a ``limiting Grothendieck ring'' $\mathcal{G}_{\infty}(\mathcal{C})$ in Section 7, and we establish the structure of $\mathcal{G}_\infty(\mathcal{C})$ in Section 8. In Section 9  we use partition combinatorics to determine explicitly how certain basis elements of the ring are expressed in terms of the the basic hooks. Finally, we discuss some applications to asymptotic representation theory of wreath products in Section 10. In Section 11 we explain how all these results generalise to the setting where $\mathcal{C}$ is a tensor category.
\subsection{Acknowledgements}
The author would like to thank Pavel Etingof for useful conversations and for comments on an earlier version of this paper, as well as the referees for their feedback and suggestions.

\section{Preliminaries}
\noindent
Throughout this paper, we work with a fixed algebraically closed field of characteristic zero, $k$, and a category $\mathcal{C} = \hopf-mod$ of finite-dimensional modules for $\hopf$, a fixed Hopf algebra over $k$. We work with $\mathcal{C}$ rather than $\hopf$ to stress that our results only depend on the module category. In fact, all our results generalise to the setting where $\mathcal{C}$ is a tensor category; we prove things in a suitable level of generality, for example we do not make use of dual modules constructed using the antipode of $\hopf$, which do not exist in a general tensor category.
\newline \newline \noindent
It is clear that if $\mathbf{1}$ is the trivial $\hopf$-module, then $\End_{\mathcal{C}}(\textbf{1}) = k$. Also, the tensor product in $\mathcal{C}$ is exact in both arguments and bilinear with respect to direct sums. The Grothendieck group, $\Groth$, has a basis (as a free abelian group) consisting of isomorphism classes of simple $\hopf$-modules. The exactness of the tensor product implies that it respects the relations of the Grothendieck group and therefore descends to a bilinear distributive multiplication on $\Groth$. Finally, the image of $\mathbf{1}$ in $\Groth$ is a multiplicative identity. Thus, $\Groth$ inherits the structure of a ring. This will be the main setting in which we work.

\section{Partitions and Symmetric Functions}
\noindent
We will make considerable use of partition combinatorics, which we review briefly. All the material that we will need can be found in the first chapter of \cite{macdonald}. 
\subsection{Partitions, Symmetric functions, Representations of Symmetric Groups}
We say that $\lambda = (\lambda_1, \lambda_2, \ldots, \lambda_l)$ is a partition of $n$ if $\lambda_1 \geq \lambda_2 \geq \cdots \geq \lambda_l$ are nonnegative integers summing to $n$, and we call the $\lambda_i$ the parts of the partition. Two partitions that differ only by the number of trailing zeroes are considered equivalent. The set of all partitions is denoted $\mathcal{P}$. The expression $(n, \lambda)$ is an abbreviation for $(n, \lambda_1, \lambda_2, \ldots, \lambda_l)$, which is also a partition provided $n \geq \lambda_1$. We write $\lambda \vdash n$ to mean that $\lambda$ is a partition of $n$. An alternative way of expressing $\lambda$ is $(1^{m_1} 2^{m_2} \cdots r^{m_r})$, where $m_i$ is the number of $j$ such that $\lambda_j = i$; in case it is unclear which partition we are considering, we write $m_i(\lambda)$ for the number of parts of $\lambda$ equal to $i$. The size of $\lambda$ is $|\lambda| = \lambda_1 + \lambda_2 + \cdots + \lambda_k = 1m_1 + 2m_2 + \cdots + rm_r$ (where $r = \lambda_1$ is the largest part of $\lambda$), $|\lambda|$ is the unique integer $n$ such that $\lambda \vdash n$. The length $l(\lambda)$ is the number of nonzero parts of $\lambda$, so we have $l(\lambda) = m_1 + m_2 + \cdots + m_r$. If $\lambda^{(j)}$ are partitions, we write $\cup_j \lambda^{(j)}$ for the partition $\mu$ obtained by merging all the partitions $\lambda^{(j)}$ together, so $m_i(\mu) = \sum_j m_i(\lambda^{(j)})$. We write $\varepsilon_{\lambda} = (-1)^{|\lambda| - l(\lambda)}$, and $z_\lambda = (m_1!)1^{m_1} (m_2!) 2^{m_2} \cdots (m_n!) n^{m_n}$.
\newline \newline
\noindent
Recall that the ring of symmetric functions, $\Lambda$, is defined as a (graded) inverse limit of the rings of invariants $\mathbb{Z}[x_1,x_2,\ldots,x_n]^{S_n}$, where the symmetric group acts by permuting the variables. It is freely generated as a polynomial algebra by the elementary symmetric functions $e_i$, but also by the complete symmetric functions $h_i$, so $\Lambda = \mathbb{Z}[e_1, e_2, \ldots] = \mathbb{Z}[h_1, h_2, \ldots]$. There are also power-sum symmetric functions $p_n$ which do not generate $\Lambda$ over $\mathbb{Z}$, but do satisfy $\Lambda \otimes_\mathbb{Z} \mathbb{Q} = \mathbb{Q}[p_1, p_2, \ldots]$. If we define the formal power series $E(t) = \sum_{n=0}^{\infty} e_n t^n$, $H(t) = \sum_{n=0}^{\infty} h_n t^n$ (here $e_0=h_0=1$), and $P(t) = \sum_{n=0}^{\infty} p_{n+1}t^n$, then we have the relations $H(t)E(-t) = 1$, and $\frac{E^{\prime}(t)}{E(t)} = P(-t)$. For a partition $\lambda = (\lambda_1, \lambda_2, \ldots, \lambda_k)$, we write $e_\lambda = e_{\lambda_1}e_{\lambda_2} \cdots e_{\lambda_k}$, and similarly we define $h_\lambda$ and $p_\lambda$
(so the $e_\lambda$ and $h_\lambda$ form $\mathbb{Z}$-bases of $\Lambda$). Another important family of symmetric functions are the Schur functions $s_\lambda$ (indexed by $\lambda \in \mathcal{P}$), which form a $\mathbb{Z}$-basis of $\Lambda$.
\newline \newline
\noindent
The irreducible representations of the symmetric group $S_n$ in characteristic zero are indexed by partitions $\lambda \vdash n$. They are called Specht modules and are denoted by $\mathcal{S}^{\lambda}$. Since the conjugacy classes of $S_n$ are also parametrised by partitions of $n$ via cycle type, we may write $\chi_{\mu}^{\lambda}$ for the value of the character of $\mathcal{S}^\lambda$ on an element of cycle type $\mu$. This allows us to express the Schur function $s_\lambda$ in terms of power-sum symmetric function as follows:
\[
s_\lambda = \sum_{\mu \vdash |\lambda|} \frac{\chi_{\mu}^{\lambda}p_\mu}{z_\mu}
\]
Since $h_n = s_{(n)}$ and $e_n = s_{(1^n)}$, we have:
\[
h_n = \sum_{\lambda \vdash n} \frac{p_\lambda}{z_\lambda}, \hspace{20mm}
e_n = \sum_{\lambda \vdash n} \frac{\varepsilon_\lambda p_\lambda}{z_\lambda}
\]
There is a nondegenerate bilinear form $\langle -, - \rangle$ on $\Lambda$ for which the Schur functions are orthonormal. It satisfies $\langle p_\lambda, p_\mu \rangle = \delta_{\lambda , \mu} z_\lambda$, where $\delta_{\lambda, \mu}$ is the Kronecker delta.

\subsection{Sets of Variables for Symmetric Functions and Some Stability Properties}
One can think of elements of $\Lambda \otimes \Lambda$ as symmetric functions that are symmetric in two sets of variables separately. We write $f(\mathbf{x})$ to indicate that $f$ is a symmetric function in the set of variables $\{ x_i\}$ (we will suppress the index set of the variables), or $f(\mathbf{x},\mathbf{y})$ to mean that $f$ is a symmetric function in the set of variables $\{ x_i\} \cup \{ y_j\}$. Similarly, we write $f(\mathbf{x}\mathbf{y})$ when the variable set is $\{ x_i y_j \}$ (for example, $p_n(\mathbf{x}, \mathbf{y}) = p_n(\mathbf{x}) + p_n(\mathbf{y})$ and $p_n(\mathbf{x}\mathbf{y}) = p_n(\mathbf{x})p_n(\mathbf{y})$). With this in mind, we have the Cauchy identity:
\[
\prod_{i,j} \frac{1}{1-x_i y_j} = \sum_{\lambda \in \mathcal{P}} s_\lambda(\mathbf{x})s_\lambda(\mathbf{y})
= \sum_{\mu \in \mathcal{P}} \frac{p_\mu(\mathbf{x}) p_\mu(\mathbf{y})}{z_\mu}
\]
We also note that $s_\lambda(\mathbf{x}\mathbf{y}) = \sum_{\mu, \nu \in \mathcal{P}} k_{\mu,\nu}^{\lambda}s_\mu(\mathbf{x}) s_\nu(\mathbf{y})$, where $k_{\mu, \nu}^{\lambda}$ are the Kronecker coefficients, defined for $|\lambda| = |\mu| = |\nu|$ as multiplicities in tensor products of Specht modules (and taken to be zero otherwise):
\[
\mathcal{S}^{\mu} \otimes \mathcal{S}^\nu = \bigoplus_{\lambda} \left( \mathcal{S}^\lambda \right)^{\oplus k_{\mu,\nu}^{\lambda}}
\]
On the other hand, $s_{\lambda}(\mathbf{x},\mathbf{y}) = \sum_{\mu, \nu} c_{\mu, \nu}^\lambda s_\mu(\mathbf{x}) s_\nu(\mathbf{y})$, where $c_{\mu, \nu}^{\lambda}$ are the Littlewood-Richardson coefficients, which also satisfy the property that $s_\mu s_\nu = \sum_{\lambda} c_{\mu, \nu}^{\lambda}s_\lambda$. The Littlewood-Richardson coefficient $c_{\mu, \nu}^{\lambda}$ is taken to be zero if $|\mu| + |\nu| \neq |\lambda|$.
\begin{definition}
Suppose that $\lambda$ is a partition, and $n$ is an integer such that $n - |\lambda| \geq \lambda_1$. We write $\lambda[n] = (n-|\lambda|, \lambda)$ for the partition obtained by adding a part at the beginnning of $\lambda$ such that the total size is $n$. (If $n - |\lambda| < \lambda_1$, we leave $\lambda[n]$ undefined.)
\end{definition}
\noindent
The Kronecker coefficients famously satisfy the following stability property \cite{murnaghan}.
\begin{lemma} \label{kron_stab}
The sequence of Kronecker coefficients $k_{\mu[n], \nu[n]}^{\lambda[n]}$ eventually becomes constant, and the stable limit, called the reduced Kronecker coefficient, is denoted $\tilde{k}_{\mu, \nu}^{\lambda}$.
\end{lemma}
\noindent
The following result of \cite{dvir} shows that in special cases, more can be said.
\begin{lemma} \label{dvir_result}
If $|\mu| + |\nu| = |\lambda|$, then the reduced Kronecker coefficient agrees with the Littlewood-Richardson coefficient: $\tilde{k}_{\mu, \nu}^{\lambda} = c_{\mu, \nu}^{\lambda}$. If $|\mu| + |\nu| < |\lambda|$, then $\tilde{k}_{\mu,\nu}^{\lambda} = 0$.
\end{lemma}
\begin{proposition}\label{LRS}
Suppose that for a partition $\lambda$ we write $\lambda^{*m}$ for the partition obtained by adding $m$ to $\lambda_1$. Then if $\lambda$, $\mu$, $\nu$ are partitions, then the sequence of Littlewood-Richardson coefficients $c_{\mu, \nu}^{\lambda}, c_{\mu^{*1}, \nu}^{\lambda^{*1}}, c_{\mu^{*2}, \nu}^{\lambda^{*2}}, \cdots$ is eventually constant.
\end{proposition}
\begin{proof}
Using the Littlewood-Richardson rule, it suffices to count the number of skew tableaux of shape $\lambda^{*m}/\mu^{*m}$ and weight $\nu$ satisfying the lattice word condition. The diagrams are related for successive $m$ by shifting the first row. We illustrate this with an example. Suppose $\lambda = (5,4,1)$, $\mu = (3,1)$ and $\nu = (3,3)$. An example of a skew-tableau of shape $\lambda/\mu$ and weight $\nu$ is
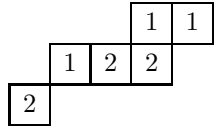
\begin{figure}[H]
\centering
\begin{ytableau}
\none & \none & \none & 1 & 1 \\
\none &   1   & 2 & 2 \\
2 
\end{ytableau}
\end{figure}
\noindent
But when $\lambda^{*2} = (7,4,1)$, $\mu^{*2} = (5,1)$ and $\nu = (3,3)$ an example of a skew-tableau of shape $\lambda^{*2}/\mu^{*2}$ and weight $\nu$ is
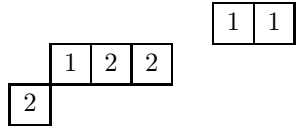
\begin{figure}[H]
\centering
\begin{ytableau}
\none & \none & \none  &\none & \none & 1 & 1 \\
\none &   1   & 2 & 2 \\
2 
\end{ytableau}
\end{figure}
\noindent
Increasing $m$ in $\lambda^{*m}/\mu^{*m}$ shifts the top row of the diagram to the right. As soon as $m$ is large enough that the first row of the skew diagram of shape $\lambda^{*m}/\mu^{*m}$ is disconnected from the rest of the diagram then the operation of further shifting the top row to the right leads to an obvious bijection of skew tableaux (the disconnected condition guarantees that the tableau property is unaffected by shifting the top row). Since it also preserves the lattice words associated to the tableaux, the lattice word property is also preserved by this row-shifting bijection. Counting the number of such tableaux gives the Littlewood-Richardson coefficient $c_{\mu^{*m}, \nu}^{\lambda^{*m}}$, which gives the result.
\end{proof}

\section{Wreath Products}
\noindent
We outline some features of wreath products that are important for us.
\subsection{Construction of $\Wreath{n}$ and Restriction/Induction}
\begin{definition}
Let $\hopf \wr S_n$ be the algebra isomorphic to $\hopf^{\otimes n} \otimes kS_n$ as a vector space, with multiplication defined as follows. Suppose that $a_i$ and $b_i$, $i \in \{1,2,\ldots,n\}$ are elements of $\hopf$, whilst $\sigma$ and $\rho$ are elements of $S_n$. Then:
\[
\left((a_1 \otimes a_2 \otimes \cdots \otimes a_n) \otimes \sigma\right) 
\left((b_1 \otimes b_2 \otimes \cdots \otimes b_n) \otimes \rho\right)
=
(a_1 b_{\sigma^{-1}(1)} \otimes a_2b_{\sigma^{-1}(2)} \otimes \cdots \otimes a_nb_{\sigma^{-1}(n)}) \otimes (\sigma \rho)
\]
It is well known that this algebra naturally inherits the structure of a Hopf algebra from the Hopf algebra structure on $\hopf$ in the following way. The maps $\hopf \to \hopf \wr S_n$ obtained by embedding $\hopf$ into the $r$-th tensor factor of $\hopf^{\otimes n} \otimes kS_n$ ($r=1,2,\ldots,n$) are maps of Hopf algebras, and similarly, the embedding $kS_n \to \hopf^{\otimes n} \otimes kS_n$ by mapping into the final tensor factor is a map of Hopf algebras. The images of these $n+1$ maps generate $\hopf \wr S_n$, and therefore determine the comultiplication. The \emph{wreath product category} $\Wreath{n}$ is the category $(\hopf \wr S_n)-mod$. We suppress $\hopf$ in the notation because $\Wreath{n}$ can be constructed from $\mathcal{C} = \hopf-mod$ alone (see the final section for details).
\end{definition}
\noindent
In our situation, it will be important to consider actions of subgroups of $S_n$. In the above, we may form the Hopf subalgebra $\hopf \wr G = \hopf^{\otimes n} \otimes kG$ for any subgroup $G$ of $S_n$. If $H$ is a subgroup of $G$, we have a restriction functor $\Res{H}{G}: \hopf \wr G-\mbox{mod} \to \hopf \wr H-\mbox{mod}$. Additionally there is an induction functor $\Ind{H}{G}: \hopf \wr H-\mbox{mod} \to \hopf \wr G-\mbox{mod}$ which is both right adjoint and left adjoint to $\Res{H}{G}$. The induction functor may be written as a sum over coset representatives of $H$ in $G$ as follows:
\[
\Ind{H}{G}(M) = \bigoplus_{g \in G/H} gM
\]
In the above formula, $gM$ denotes an object isomorphic to $M$ as a vector space, and the action of $G$ is that of an induced representation. Explicitly, to see how $x\in G$ acts on $gM$, note that $xg = g^\prime h$ for unique coset representative $g^\prime \in G/H$ and $h \in H$. Then, $x$ takes $gM$ to $g^\prime M$, whilst acting by the usual action of $h \in H$. Because induction and restriction are exact functors, they define homomorphisms between the Grothendieck groups of $\hopf \wr G$ and $\hopf \wr H$. We will be interested in the case where $G=S_n$ and $H$ is a subgroup of the following type.
\begin{definition}
If $\alpha = (\alpha_1, \alpha_2, \ldots, \alpha_k)$ is a composition of $n$ (that is, a finite sequence of nonnegative integers summing to $n$), then $S_\alpha = \prod_{i=1}^k S_{\alpha_i}$ is the \emph{Young subgroup} of $S_n$ corresponding to the composition $\alpha$ (i.e. $Sym(\{1,2,\ldots,\alpha_1\}) \times Sym(\{\alpha_1+1,\ldots,\alpha_1+\alpha_2\})\times \cdots \times Sym(\{n - \alpha_k + 1, \ldots, n\})$). We refer to the factors $S_{\alpha_i}$ as the component groups of $S_\alpha$.
\end{definition}
\noindent
Note that for $\alpha = (\alpha_1, \alpha_2, \ldots, \alpha_k)$, we may identify $\hopf \wr S_\alpha = \hopf^{\otimes n} \otimes kS_\alpha$ with $ \bigotimes_{i} \left(\hopf^{\otimes \alpha_i} \otimes kS_{\alpha_i}\right) = \bigotimes_i (\hopf \wr S_{\alpha_i})$.

\subsection{Description of Simple Objects in $\Wreath{n}$}
\begin{definition}
Suppose that $M$ and $N$ are modules over algebras $A$ and $B$ respectively. Then we write $M \boxtimes N$ for $M \otimes N$ viewed as an $A \otimes B$-module. We also write $M^{\boxtimes n}$ for $M^{\otimes n}$ viewed as an $A^{\otimes n}$-module.
\end{definition}
\noindent
In the sequel we will often consider objects of the following form. 
\begin{definition}
Let $M \in \mathcal{C}$, and $V$ be a finite-dimensional representation of $S_n$ over $k$. We define the following object of $\Wreath{n}$.
\[
M^{\boxtimes n} \otimes  V
\]
This has a $\hopf^{\otimes n}$-action by acting on the first tensor factor. An element of $S_n$ acts by permuting the factors of $M^{\boxtimes n}$ in the obvious way, as well as acting on $V$.
\end{definition}
\noindent
We now introduce some standard properties of the $M^{\boxtimes n} \otimes V$. The proofs of most of these statements are well known, and therefore omitted.
\begin{proposition} \label{Wprop1}
Let $V_1, V_2$ be finite-dimensional representations of $S_n$, and let $M, M^{\prime}$ be objects of $\mathcal{C}$. We have the following.
\begin{enumerate}
\item \[
M^{\boxtimes n} \otimes (V_1 \oplus V_2) \cong (M^{\boxtimes n} \otimes V_1) \oplus (M^{\boxtimes n} \otimes V_2)
\]
\item \[
(M^{\boxtimes n} \otimes V_1) \otimes (M^{\prime \boxtimes n} \otimes V_2) \cong (M \otimes M^\prime)^{\boxtimes n} \otimes (V_1 \otimes V_2)
\]
\end{enumerate}
\end{proposition}
\noindent
When considering Mackey theory, it will also be necessary to understand the behaviour of $M^{\boxtimes n} \otimes V$ under induction and restriction.
\begin{proposition} \label{Wprop2}
\begin{enumerate}
\item Suppose that $M$ is an object of $\mathcal{C}$, $V_1$ is a finite-dimensional representation of $S_{n_1}$, and $V_2$ is a finite-dimensional representation of $S_{n_2}$. Then:
\[
\Ind{S_{n_1} \times S_{n_2}}{S_{n_1+n_2}}((M^{{\boxtimes} n_1} \otimes V_1) \boxtimes (M^{{\boxtimes} n_2} \otimes V_2)) = M^{{\boxtimes} (n_1+n_2)} \otimes \Ind{S_{n_1} \times S_{n_2}}{S_{n_1+n_2}}(V_1 \boxtimes V_2)
\]
Note that the induction on the left relates to modules for wreath products, while the induction on the right relates to modules for symmetric groups.

\item Suppose that $M$ is an object of $\mathcal{C}$, and $V$ is a finite-dimensional representation of $S_n$ such that $\Res{S_{n_1} \times S_{n_2}}{S_n}(V) = \bigoplus_i V_1^{(i)} \boxtimes V_{2}^{(i)}$ (where $V_j^{(i)}$ is a representation of $S_{n_j}$ for $j=1,2$). Then:
\[
\Res{S_{n_1} \times S_{n_2}}{S_n}(M^{\boxtimes n} \otimes V) = \bigoplus_i (M^{\boxtimes n_1} \otimes V_1^{(i)}) \boxtimes (M^{\boxtimes n_2} \otimes V_2^{(i)})
\]
\end{enumerate}
\end{proposition}
\begin{definition}
Let $I(\mathcal{C})$ be the set of isomorphism classes of simple objects of $\mathcal{C}$. Let $\mathcal{G}(\Wreath{n})$ denote the Grothendieck group of $\Wreath{n}$, and for an object $R$ of $\Wreath{n}$, let $[R]$ denote the image of $R$ in $\mathcal{G}(\Wreath{n})$. As in \cite{Nate}, let:
\[
\mathcal{P}_{n}^{\mathcal{C}} = \{ \vv{\lambda} : I(\mathcal{C}) \to \mathcal{P} \mid \sum_{U \in I(\mathcal{C})} |\vv{\lambda}(U)| = n\}
\]
Thus, $\mathcal{P}_n^{\mathcal{C}}$ is the set of multipartitions of $n$ whose constituent partitions are indexed by isomorphism classes of simple objects in $\mathcal{C}$. We will indicate multipartitions (elements of $\mathcal{P}_n^{\mathcal{C}}$ for some $n$) with arrows (e.g.  $\vv{\lambda}, \vv{\mu}, \vv{\nu}$), while ordinary partitions (elements of $\mathcal{P}$) will not have arrows (e.g. $\lambda, \mu, \nu, \rho, \sigma, \tau$).
\end{definition}
\noindent
Eventually, we will pass to the Grothendieck group of $\Wreath{n}$, and we will wish to understand the composition factors of $M^{\boxtimes n} \otimes V$. The following proposition will allow us to calculate the composition factors that we will be interested in. The proof is routine, and we omit it.
\begin{proposition} \label{Wprop3}
Suppose $N$ is a subobject of $M$ in $\Wreath{n}$. If $\mathbf{1}_G$ denotes the trivial representation of a group $G$, we have the following equality in $\mathcal{G}(\Wreath{n})$:
\[
[M^{\boxtimes n} \otimes \mathbf{1}_{S_n}] = \sum_{r=0}^{n} [\Ind{S_r \times S_{n-r}}{S_n}\left((N^{\boxtimes r} \otimes \mathbf{1}_{S_r}) \boxtimes ((M/N)^{\boxtimes (n-r)} \otimes \mathbf{1}_{S_{n-r}})\right)]
\]
\end{proposition}
\noindent
We now describe the simple objects in the category $\Wreath{n}$. The set $\mathcal{P}_{n}^{\mathcal{C}}$ gives an index set for the isomorphism classes of simple objects of $\Wreath{n}$.
\begin{definition}
Let $\vv{\lambda} \in \mathcal{P}_n^{\mathcal{C}}$, and $K = \prod_{U \in I(\mathcal{C})} S_{|\vv{\lambda}(U)|}$, a Young subgroup of $S_n$. We define $R_{\vv{\lambda}}$, an object of $\Wreath{n}$:
\[
R_{\vv{\lambda}} = \Ind{K } {S_n}\left( \boxtimes_{U \in I(\mathcal{C})} \left(U^{\boxtimes |\vv{\lambda}(U)|} \otimes \mathcal{S}^{\vv{\lambda}(U)}   \right)\right)
\]
As before, $\mathcal{S}^{\mu}$ denotes the Specht module associated to an integer partition $\mu$.
\end{definition}
\noindent
The $R_{\vv{\lambda}}$ are the (pairwise non-isomorphic) simple objects of $\Wreath{n}$. In \cite{Mori}, this is shown in the context of indecomposable objects of an additive category, but the proof in our setting is analogous. We will use Mackey theory to calculate tensor products of the $R_{\vv{\lambda}}$, and hence the multiplicative structure of the Grothendieck ring $\mathcal{G}(\Wreath{n})$.

\section{Mackey Theory}
\begin{definition}
Suppose that $H$ and $K$ are subgroups of a group $G$. A $(H,K)$-double coset in $G$ is an orbit of the group action of $H \times K$ on $G$ given by $(h, k) \cdot g = hgk^{-1}$. A double coset representative is any element of a double coset, and we write $H \backslash G / K$ for a set of $(H, K)$-double coset representatives in $G$ (a set containing one representative from each $(H, K)$-double coset in $G$).
\end{definition}
\noindent
The following results are completely analogous to the corresponding versions for representations of finite groups, including the proofs, which we omit (see for example, \cite{repthyZ}).
\begin{proposition}
Let $G$ be a finite group with subgroups $H$ and $K$. Suppose that $M$ is a $K$-equivariant object of a $k$-linear abelian category with a $G$-action. We have the following formula for the composition of induction and restriction.
\[
\Res{H}{G}(\Ind{K}{G}(M)) \cong \bigoplus_{s \in H \backslash G / K} \Ind{H \cap sKs^{-1}}{H}(\Res{H \cap sKs^{-1}}{sKs^{-1}}(sM))
\]
\end{proposition}
\begin{lemma} \label{ind_res_tensor}
If $H$ is a subgroup of the finite group $G$, and $M, N$ are objects of a monoidal $k$-linear abelian category such that $M$ is $H$-equivariant and $N$ is $G$-equivariant, then we have $\Ind{H}{G}(M) \otimes N \cong \Ind{H}{G}(M \otimes \Res{H}{G}(N))$. Similarly $N \otimes \Ind{H}{G}(M) \cong \Ind{H}{G}(\Res{H}{G}(N) \otimes M)$
\end{lemma}
\begin{proposition} \label{mackey_product}
Suppose that $H$ and $K$ are subgroups of the finite group $G$. If $M$ is an $H$-equivariant object of a monoidal $k$-linear abelian category with an action of $G$, and $N$ is a $K$-equivariant object, we have the following.
\[
\Ind{H}{G}(M) \otimes \Ind{K}{G}(N) \cong \bigoplus_{s \in H \backslash G / K}  \Ind{H \cap sKs^{-1}}{G}(\Res{H \cap sKs^{-1}}{H}(M) \otimes \Res{H \cap sKs^{-1}}{sKs^{-1}}(sN))
\]
\end{proposition}
\begin{proof}
\begin{eqnarray*}
\Ind{H}{G}(M) \otimes \Ind{K}{G}(N) &\cong& \Ind{H}{G}(M \otimes \Res{H}{G}(\Ind{K}{G}(N))) \\
&\cong& \Ind{H}{G}(M \otimes \bigoplus_{s \in H \backslash G / K} \Ind{H \cap sKs^{-1}}{H}(\Res{H \cap sKs^{-1}}{sKs^{-1}}(sN)) ) \\
&\cong&  \bigoplus_{s \in H \backslash G / K} \Ind{H}{G}(M \otimes  \Ind{H \cap sKs^{-1}}{H}(\Res{H \cap sKs^{-1}}{sKs^{-1}}(sN)) ) \\
&\cong&  \bigoplus_{s \in H \backslash G / K} \Ind{H}{G}( \Ind{H \cap sKs^{-1}}{H}(\Res{H \cap sKs^{-1}}{H}(M) \otimes \Res{H \cap sKs^{-1}}{sKs^{-1}}(sN)) ) \\
&\cong&  \bigoplus_{s \in H \backslash G / K}  \Ind{H \cap sKs^{-1}}{G}(\Res{H \cap sKs^{-1}}{H}(M) \otimes \Res{H \cap sKs^{-1}}{sKs^{-1}}(sN)) 
\end{eqnarray*}
Here we have used the transitivity of induction, namely that $\Ind{H}{G} \circ \Ind{H \cap sKs^{-1}}{H} = \Ind{H \cap sKs^{-1}}{G}$ (the proof of this is again analogous to the proof for representations of finite groups).
\end{proof}
\noindent
In our setting, $H$ and $K$ will be Young subgroups of $S_n$ (recall that the simple objects of the wreath product category are obtained by applying induction functors from Young subgroups). So if we are to use the previous proposition to decompose tensor products of simple objects, it will be important to understand double coset representatives of Young subgroups of $S_n$.

\section{Double Cosets of Young Subgroups}
\noindent
We now prove some facts about minimal length double coset representatives of Young subgroups of symmetric groups. Let $\sigma\in S_n$ be considered as a bijective function from the set $\{1, 2, \ldots , n\}$ to itself. If $\mu = (\mu_1, \mu_2, \ldots)$ and $\nu = (\nu_1, \nu_2, \ldots)$ are compositions of $n$ and $S_\mu = \prod_i S_{\mu_i}, S_\nu = \prod_i S_{\nu_i}$ are the associated Young subgroups of $S_n$, we seek to describe the $(S_\mu, S_\nu)$-double cosets of $S_n$. We write $A_i$ for the subset of $\{ 1, 2, \ldots ,n \}$ that is permuted by $S_{\mu_i}$ (considered as a subgroup of $S_\mu$), so that $A_1 = \{1,2, \ldots, \mu_1\}$, $A_2 = \{\mu_1 + 1, \mu_1 + 2, \ldots, \mu_1 + \mu_2 \}$ and so on. Similarly we define $B_i$ to be the subsets of $\{1, 2, \ldots, n \}$ permuted by the $S_{\nu_i}$.

\begin{definition}
Say that $\sigma \in S_n$ is \textit{fully ordered} if (in the notation defined above), for each $B_i$ and $A_j$, the restrictions $\sigma |_{B_i}$ and ${\sigma}^{-1}|_{A_j}$ are monotone increasing functions.
\end{definition}
\begin{remark}
The property of being fully ordered will turn out to be equivalent to being a minimal length $(S_\mu, S_\nu)$-double coset representative. However, it will be more convenient to work with the above definition.
\end{remark}

\begin{lemma}
The numbers $C_{i,j}(\sigma) = |\{x \in B_i | \sigma(x) \in A_j\}|$ are double coset invariants. Moreover, $\sigma_1$ and $\sigma_2$ are in the same double coset if and only if $C_{i,j}(\sigma_1) = C_{i,j}(\sigma_2)$ for all $i,j$. Additionally, each double coset has a unique fully ordered element.
\end{lemma}
\begin{proof}
The $C_{i,j}(\sigma)$ are constant on each double coset since both the left ($S_\mu$) and right ($S_\nu$) actions preserve each $A_j$ and $B_i$. We show that every element can be acted on by the left by $S_\mu$ and on the right by $S_\nu$ to obtain a totally ordered element. Then we show that a totally ordered element is determined by the $C_{i,j}(\sigma)$. This implies that if two elements of $S_n$ have the same $C_{i,j}(\sigma)$, then they have the same fully ordered element in their double coset, implying that they are in the same coset.
\newline \newline \noindent
Given $\sigma \in S_n$, we may act on it on the right by elements of the $S_{\nu_i}$ (and hence an element of $S_\nu$) to reorder the elements of $B_i$ in order of increasing image under $\sigma$. This gives a new $\sigma \in S_n$ such that if $x < y$ are elements of $B_i$, then ${\sigma}(x) < {\sigma}(y)$. This means that the restriction of ${\sigma}$ to each $B_i$ is a monotone increasing function. We may also act on the left by the $S_{\mu_i}$ to sort the preimages of each $A_j$; thus we may assume if $x \in B_a$ and $y \in B_b$ with $a < b$, such that $\sigma(x)$ and ${\sigma}(y)$ are in $A_i$, then $\sigma(x) < \sigma(y)$. Note that this process preserves the property that $\sigma$ is monotone increasing when restricted to the $B_i$. Thus the result of these actions is a fully ordered element.
\newline \newline \noindent
Next we inductively construct a fully ordered $\sigma$ from prescribed $C_{i,j}$. For a collection of natural numbers $C_{i,j}^{\prime}$, there is a $\sigma \in S_n$ such that $C_{i,j}(\sigma) = C_{i,j}^{\prime}$ if and only if the following two conditions hold. For each $j$, $\sum_i C_{i,j}^{\prime} = |B_j|$ and for each $i$, $\sum_j C_{i,j}^{\prime} = |A_i|$.
The first $C_{1,1}(\sigma)$ elements of $B_1$ must map to the the first $C_{1,1}(\sigma)$ elements of $A_1$ (in a monotone increasing way, hence uniquely). Then, the next $C_{1,2}(\sigma)$ elements of $B_1$ map to the first $C_{1,2}(\sigma)$ elements of $A_2$, and so on. This means that the image of $B_1$ is determined uniquely. Then, the first $C_{2,1}(\sigma)$ elements of $B_2$ map to the next $C_{2,1}(\sigma)$ elements of $A_1$, and so on (again without choice). Repeating for all $i$ and $j$, we obtain a fully ordered element $\sigma$ and each step of the construction was forced, so the fully ordered element is unique.
\end{proof}
\begin{remark}
Noting that the length of $\sigma \in S_n$ is equal to the number of inversions (that is, pairs $(i,j)$ with $1 \leq i < j \leq n$ such that $\sigma(j) < \sigma(i)$), the property of being fully ordered is the same as being a minimal length double coset representative. Note that the number of inversions is bounded below by $\sum_{i_1<i_2, j_1 < j_2} |C_{i_1, j_2}(\sigma)||C_{i_2,j_1}(\sigma)|$, and a fully ordered $\sigma$ attains this bound.
\end{remark}
\noindent
For convenience, in this section we require that for a composition $\alpha$, the factors in the Young subgroup $S_\alpha = \prod_i S_{\alpha_i}$ are ordered in increasing order from left to right, i.e. $S_{\alpha_k} \times \cdots \times S_{\alpha_1}$. For example, $S_{(3,2,1)} = S_1 \times S_2 \times S_3$. We will be interested in the operation of increasing the largest part of a partition by $1$ (hence passing from partitions of $n$ to partitions of $n+1$). 
\begin{definition}
If $\alpha$ is a composition of $n$, write $\alpha^*$ for composition of $n+1$ obtained by adding $1$ to the first part of $\alpha$. 
\end{definition}
\noindent
Correspondingly, we discuss $(S_\mu, S_\nu)$-double coset representatives under the operation of adding $1$ to the largest parts of the partitions $\mu$ and $\nu$. If $f$ is a bijection from the set $\{ 1,2, \ldots, n \}$ to itself satisfying the fully ordered property, it continues to satisfy the fully ordered property as a function on $\{1,2,\ldots, n+1\}$ when we define $f(n+1) = n+1$ (note that this corresponds to the inclusion $S_n \hookrightarrow S_{n+1}$ by considering elements fixing $n+1$).
\begin{proposition} \label{cosetStab}
Let $\mu,\nu$ be partitions of $n$. After sufficiently many repeated applications of the operation $(\mu, \nu, n) \mapsto (\mu^*, \nu^*, n+1)$, the number of $(S_{\mu},S_{\nu})$-double cosets in $S_n$ stabilises. Moreover, one can choose representatives which are identified for different $n$ (sufficiently large) via the usual inclusions of symmetric groups.
\end{proposition}
\begin{proof}
Observe that if the first parts of $\mu^*$ and $\nu^*$ each exceed $n/2$, then $C_{(\mu^*)_1, (\nu^*)_1} \geq 1$ by the pigeonhole principle. This means that for a fully ordered double coset representative $\sigma$, $\sigma(n+1) = n+1$. In particular, each double coset representative is obtained from a double coset representative of $(S_\mu, S_\nu)$ under the inclusion of $S_n$ in $S_{n+1}$.
\end{proof}

\begin{remark}
Ordering the multiplicative factors in the definition of Young subgroup from smallest to largest allows us to take the inclusions $S_n \hookrightarrow S_{n+1}$ obtained by extending functions on $\{1, 2, \ldots, n\}$ by requiring them to fix $n+1$. If we did not do this, we would have to use a nonstandard inclusion. The Young subgroups related by different orderings of their component groups are conjugate in $S_n$, so induced objects coming from the two embeddings are related by a twist which will be irrelevant for our purposes.
\end{remark}

\section{The Limiting Grothendieck Ring}
\noindent
We work towards understanding the tensor product, with the aim of constructing a ``stable'' Grothendieck ring.
\subsection{Tensor Products of Irreducible Modules}
Recall that $\mathcal{G}(\Wreath{n})$ naturally inherits the structure of a ring, and has a basis given by $[R_{\vv{\lambda}}]$ for $\vv{\lambda} \in \mathcal{P}_n^\mathcal{C}$.
\begin{example}
If $\mathcal{C}$ is the category of finite-dimensional vector spaces over $k$ (e.g. if $\mathcal{R} = k$), $\mathcal{G}(\Wreath{n})$ is the representation ring of $S_n$ over $k$.
\end{example}
\noindent
Given $\vv{\mu}, \vv{\nu} \in \mathcal{P}_n^{\mathcal{C}}$, we wish to describe $[R_{\vv{\mu}}][R_{\vv{\nu}}] = [R_{\vv{\mu}} \otimes R_{\vv{\nu}}]$ as a linear combination of $[R_{\vv{\lambda}}]$. For this task, we use the categorical Mackey theory results. We write $H_{\vv{\lambda}} = \prod_{U \in I(\mathcal{C})}S_{|\vv{\lambda}(U)|}$ for the subgroup of $S_n$ from which $R_{\vv{\lambda}}$ is induced. Note that $H_{\vv{\lambda}}$ is itself a Young subgroup of $S_n$.
\begin{lemma} \label{tensorR}
Using Proposition \ref{mackey_product}, we have the following:
\begin{align*}
R_{\vv{\mu}} \otimes R_{\vv{\nu}} &\cong \Ind{H_{\vv{\mu}}}{S_n}\left(\boxtimes_{U \in I(\mathcal{C})} \left(U^{\boxtimes |\vv{\mu}(U)|} \otimes \mathcal{S}^{\vv{\mu}(U)} \right)\right)
 \otimes \Ind{H_{\vv{\nu}}}{S_n}\left(\boxtimes_{U \in I(\mathcal{C})} \left(U^{\boxtimes |\vv{\nu}(U)|} \otimes \mathcal{S}^{\vv{\nu}(U)} \right) \right) \\
 &\cong \bigoplus_{t \in H_{\vv{\mu}} \backslash S_n / H_{\vv{\nu}}} \Ind{H_{\vv{\mu}} \cap tH_{\vv{\nu}}t^{-1}}{S_n}\left(
 \Res{H_{\vv{\mu}} \cap tH_{\vv{\nu}}t^{-1}}{H_{\vv{\mu}}}\left(\boxtimes_{U \in I(\mathcal{C})} \left( U^{\boxtimes |\vv{\mu}(U)|} \otimes \mathcal{S}^{\vv{\mu}(U)} \right) \right)\right. \\
 &  \hspace{36mm}
 \otimes \hspace{0mm}
 \Res{H_{\vv{\mu}} \cap tH_{\vv{\nu}}t^{-1}}{tH_{\vv{\nu}}t^{-1}}\left. \left(t\boxtimes_{U \in I(\mathcal{C})} \left(U^{\boxtimes |\vv{\nu}(U)|} \otimes \mathcal{S}^{\vv{\nu}(U)} \right)\right)  \right) 
\end{align*}
\end{lemma}
\begin{remark} \label{contiguityRmk}
Observe that $H_{\vv{\mu}} \cap tH_{\vv{\nu}}t^{-1}$ is a proper subgroup of at least one of $H_{\vv{\mu}}$ and $tH_{\vv{\nu}}t^{-1}$ unless these two are equal. Later on, this observation will imply the vanishing of certain restrictions of virtual representations. We also note that by the fully ordered property of $t$, $H_{\vv{\mu}} \cap tH_{\vv{\nu}}t^{-1}$ is a Young subgroup of $S_n$; it independently permutes contiguous subsets of $\{1,2, \ldots, n\}$.
\end{remark}

\noindent
We now exploit the stability property of double cosets of Young subgroups to define a ``limiting Grothendieck group''. 

\begin{definition}
Let $\mathbf{1}$ be the unit object of $\mathcal{C}$. If $\vv{\lambda} \in \mathcal{P}_n^{\mathcal{C}}$, write $\vv{\lambda}^{*}$ for the element of $\mathcal{P}_{n+1}^{\mathcal{C}}$ obtained by adding $1$ to the largest part of $\vv{\lambda}(\mathbf{1})$. We denote the result of applying the operation $\vv{\lambda} \mapsto \vv{\lambda}^{*}$ successively $k$ times by $\vv{\lambda}^{*k}$.
\end{definition}
\noindent
Specifically, we will consider products $[R_{{\vv{\mu}}^{*k}}][R_{{\vv{\nu}}^{*k}}]$ for $\vv{\mu}, \vv{\nu} \in \mathcal{P}_n^{\mathcal{C}}$ (for some $n$) as $k \to \infty$. We first introduce notation to conveniently describe the limit.
\begin{definition}
If $\vv{\lambda} \in \mathcal{P}_n^\mathcal{C}$, we define the multipartition $\vv{\lambda}^{\prime} \in \mathcal{P}_{n - \vv{\lambda}(\mathbf{1})_1}^{\mathcal{C}}$ by removing the largest part of the partition $\vv{\lambda}(\mathbf{1})$ (if this was the empty partition, it remains the empty partition). Also, if $\vv{\lambda} \in \mathcal{P}_m^\mathcal{C}$, write $\vv{\lambda}[n]$ for the element of $\mathcal{P}_n^\mathcal{C}$ obtained by appending a part of size $n-m$ to the start of the partition $\vv{\lambda}(\mathbf{1})$; this is only defined if $n-m \geq \vv{\lambda}(\mathbf{1})_1$. Explicitly, for all $U$ different from $\mathbf{1}$, $\vv{\lambda}^{\prime}(U) = \vv{\lambda}[n](U) = \vv{\lambda}(U)$, and $\vv{\lambda}^{\prime}(\mathbf{1}) = \vv{\lambda}(\mathbf{1})\backslash(\vv{\lambda}(\mathbf{1})_1)$, whilst for $n-m \geq \vv{\lambda}(\mathbf{1})_1$, $\vv{\lambda}[n](\mathbf{1}) = (n-m, \vv{\lambda}(\mathbf{1}))$. We leave the operation undefined if the inequality does not hold. Finally, we define the following set which will index a basis of the limiting Grothendieck ring.
\[
\mathcal{P}^\mathcal{C} = \bigcup_{n \geq 0} \mathcal{P}_n^\mathcal{C} = \{\lambda : I(\mathcal{C}) \to \mathcal{P} \mid \sum_{U \in I(C)} |\lambda(U)| < \infty \}
\]
\end{definition}
\noindent
The above operations satisfy some trivial properties.
\begin{lemma}
The operations $\vv{\lambda} \mapsto \vv{\lambda}^{\prime}$ and $\vv{\lambda} \mapsto \vv{\lambda}[n]$ satisfy the following relations.
\begin{enumerate}
\item $\vv{\lambda}^{\prime} = \vv{\lambda}^{*\prime}$
\item If $\vv{\lambda} \in \mathcal{P}_n^\mathcal{C}$, then $\vv{\lambda^{\prime}}[n] = \vv{\lambda}$.
\item If $\vv{\lambda} \in \mathcal{P}_m^\mathcal{C}$ and $(n-m) \geq \vv{\lambda}(\mathbf{1})_1$, then $\vv{\lambda}[n]^{\prime} = \vv{\lambda}$. In particular, this holds for $n$ sufficiently large.
\item $\bigcup_{n \in \mathbb{Z}_{\geq 0}} \{ \vv{\lambda}^{\prime} \mid \vv{\lambda} \in \mathcal{P}_n^{\mathcal{C}}\} = \mathcal{P}^{\mathcal{C}}$
\end{enumerate}
\end{lemma}

\begin{definition}
Given $\vv{\mu}, \vv{\nu} \in \mathcal{P}^{\mathcal{C}}$, we may write (in $\mathcal{G}(\Wreath{n})$ for $n$ sufficiently large)
\[
[R_{\vv{\mu}[n]}][R_{\vv{\nu}[n]}]
= \sum_{\vv{\lambda} \in \mathcal{P}^{\mathcal{C}}} c_{\vv{\mu}, \vv{\nu}}^{\vv{\lambda}}(n) [R_{\vv{\lambda}[n]}]
\]
In the above sum, we only consider terms for which $\vv{\lambda}[n]$ is well defined. We define the numbers $c_{\vv{\mu}, \vv{\nu}}^{\vv{\lambda}}(n) \in \mathbb{Z}_{\geq 0}$ via the preceding equation, and note that for fixed $\vv{\mu}, \vv{\nu}, \vv{\lambda}$ it is defined for all sufficiently large $n$.
\end{definition}

\noindent
We use Mackey theory to show that the $c_{\vv{\mu}, \vv{\nu}}^{\vv{\lambda}}(n)$ become constant as $n \to \infty$. Specifically, Lemma \ref{tensorR} describes the decomposition and Proposition \ref{cosetStab} implies that the index set of the sum stabilises. So, it suffices to show that for any fixed $t$ in the set of double coset representatives, the corresponding summand also stabilises:
\begin{align*} 
\Ind{H_{\vv{\mu}[n]} \cap tH_{\vv{\nu}[n]}t^{-1}}{S_n}\left( 
 \Res{H_{\vv{\mu}[n]} \cap tH_{\vv{\nu}[n]}t^{-1}}{H_{\vv{\mu}[n]}}\left(\boxtimes_{U \in I(\mathcal{C})} \left(U^{\boxtimes |\vv{\mu}[n](U)|} \otimes \mathcal{S}^{\vv{\mu}[n](U)}\right)\right) \right. \\
 \otimes \left. \Res{H_{\vv{\mu}[n]} \cap tH_{\vv{\nu}[n]}t^{-1}}{tH_{\vv{\nu}[n]}t^{-1}}\left(t \boxtimes_{U \in I(\mathcal{C})} \left(U^{\boxtimes |\vv{\nu}[n](U)|} \otimes \mathcal{S}^{\vv{\nu}[n](U)} \right)\right) \right)
\end{align*}
To demonstrate stability, we first show that the restrictions in the above expression stabilise in a particular sense.

\begin{lemma} \label{PL}
Let $n$ be sufficiently large, depending on $\vv{\mu}, \vv{\nu}$, and $t$, where $t$ is a fully ordered $(H_{\vv{\mu}[n]}, H_{\vv{\nu}[n]})$-double coset representative. There exists $g \in S_n$ (identified for different $n$ via usual inclusions of symmetric groups) such that $g(H_{\vv{\mu}[n]} \cap tH_{\vv{\nu}[n]}t^{-1})g^{-1} = S_{\sigma[n]}$, where $\sigma = (\sigma_1, \sigma_2, \ldots, \sigma_l)$ is some partition.
Additionally, the restriction
\[
 \Res{H_{\vv{\mu}[n]} \cap tH_{\vv{\nu}[n]}t^{-1}}{H_{\vv{\mu}[n]}}\left(\boxtimes_{U \in I(\mathcal{C})} \left(U^{\boxtimes |\vv{\mu}[n](U)|} \otimes \mathcal{S}^{\vv{\mu}[n](U)} \right) \right)
\]
is equal to a finite direct sum of expressions of the following form, where the multiplicity of each term does not vary with $n$, provided $n$ is sufficiently large:
\[
\left( \boxtimes_{i = 1}^{l} \left(U_i^{\boxtimes |\tau^{(i)}|} \otimes \mathcal{S}^{\tau^{(i)}} \right) \right) \boxtimes \left(\mathbf{1}^{\boxtimes (n - |\sigma|)} \otimes \mathcal{S}^{(n - |\tau^{(0)}| - |\sigma|, \tau^{(0)})}\right)
\]
Here the $U_i$ are not necessarily distinct and each $\tau^{(i)}$ is a partition of $\sigma_i$ ($\tau^{(0)}$ is arbitrary, but only finitely many cases appear).
\end{lemma}

\begin{proof}
The subgroup $H_{\vv{\mu}[n]} \cap tH_{\vv{\nu}[n]}t^{-1}$ is a Young subgroup of $S_n$ by Remark \ref{contiguityRmk}, which we may conjugate to reorder the component groups in order of increasing size. Explicitly, we have $g \in S_n$ such that $g(H_{\vv{\mu}[n]} \cap tH_{\vv{\nu}[n]}t^{-1})g^{-1} = S_{\alpha}$ for some partition $\alpha = (\sigma_0, \sigma_1, \ldots, \sigma_l)$. We now show that we may take $\sigma_1, \sigma_2, \ldots, \sigma_l$ to be constant with respect to $n$, and hence $\sigma = (\sigma_1, \sigma_2, \ldots, \sigma_l)$ is the desired partition from the statement of the lemma (and $\sigma_0 = n - |\sigma|$).
\newline \newline \noindent
By Proposition \ref{cosetStab}, we may assume the double coset representative $t$ is preserved under the inclusions $S_{n} \hookrightarrow S_{n+1} \hookrightarrow \cdots$. Hence, $t$ fixes $i$ for $i$ larger than some fixed constant $m_0$ depending on $t$. If $m_1 = |\vv{\mu}| - |\vv{\mu}(\mathbf{1})_1|$, then by construction $H_{\vv{\mu}}$ contains the symmetric group on $\{m_1 +1, m_1 + 2, \ldots, n \}$ (this is a subgroup of the component group whose representation $U^{\boxtimes |\tau|} \otimes \mathcal{S}^{\tau}$ has $U = \mathbf{1}$). Similarly if $m_2 = |\vv{\nu}| - |\vv{\nu}(\mathbf{1})_1|$, then $H_{\vv{\nu}}$ contains the symmetric group on $\{ m_2 + 1, m_2 + 2, \ldots, n\}$. Now, we let $M = \max(m_0, m_1, m_2)$, and we observe that the symmetric group on $\{M+1, M+2, \ldots, n\}$ is contained in $H_{\vv{\mu}} \cap tH_{\vv{\nu}}t^{-1}$ (it is contained in each of $H_{\vv{\mu}}$ and $H_{\vv{\nu}}$ and commutes with $t$). Thus we may choose $g$ in the previous paragraph to fix $M+1, M+2, \ldots, n$ by making the component group permuting the orbit of $n$ appear as the last factor in the construction of $S_{(\sigma_0, \sigma_1, \ldots, \sigma_l)}$ (for $n$ sufficiently large, this is consistent with our convention of ordering of component groups in a Young subgroup of a symmetric group). By similar reasoning, the other component groups of $H_{\vv{\mu}}\cap tH_{\vv{\nu}}t^{-1}$ are stable when passing from $n$ to $n+1$, meaning that $H_{\vv{\mu}}\cap tH_{\vv{\nu}}t^{-1}$ decomposes as a product of a fixed number of symmetric groups $S_{\sigma_i}$ (where $\sigma_i$ are constant with respect to $n$) and $S_{n-\sum_i \sigma_i}$. Since $g$ fixes all $i$ greater than $M$, it is compatible with the inclusions $S_{n} \hookrightarrow S_{n+1}$.
\newline \newline \noindent
Next we discuss stability of the restriction. The restriction of an external product is the same as the external product of restrictions.
\begin{eqnarray*} 
& &\Res{H_{\vv{\mu}[n]} \cap tH_{\vv{\nu}[n]}t^{-1}}{H_{\vv{\mu}[n]}}\left(\boxtimes_{U \in I(\mathcal{C})} \left(U^{\boxtimes |\vv{\mu}[n](U)|} \otimes \mathcal{S}^{\vv{\mu}[n](U)}\right)\right)\\
&\cong&
\boxtimes_{U \in I(\mathcal{C})} 
\Res{S_{|\vv{\mu}[n](U)|} \cap tH_{\vv{\nu}[n]}t^{-1}}{S_{|\vv{\mu}[n](U)|}}\left(U^{\boxtimes |\vv{\mu}[n](U)|} \otimes \mathcal{S}^{\vv{\mu}[n](U)} \right)\\
\end{eqnarray*}
Because the intersection $H_{\vv{\mu}[n]} \cap tH_{\vv{\nu}[n]}t^{-1}$ is a Young subgroup of $S_n$, it follows that for each $U$, $S_{|\vv{\mu}[n](U)|} \cap tH_{\vv{\nu}[n]}t^{-1}$, is a Young subgroup of $S_{|\vv{\mu}[n](U)|}$ and the product of these across $U \in I(\mathcal{C})$ will give $H_{\vv{\mu}[n]}\cap tH_{\vv{\nu}[n]}t^{-1}$. Because $H_{\vv{\mu}[n]}\cap tH_{\vv{\nu}[n]}t^{-1}$ is conjugate to $S_{(n-|\sigma|, \sigma)}$ by reordering of component groups, we may write $S_{|\vv{\mu}[n](U)|}\cap tH_{\vv{\nu}[n]}t^{-1} = \prod_{j \in I_U}S_{\sigma_j}$, where the $I_U$ are disjoint subsets of $\{0,1, \ldots, l\}$ indexed by $U \in I(\mathcal{C})$ that partition $\{0, 1, \ldots, l\}$. In particular, when $U \neq \mathbf{1}$, $\prod_{j \in I_U} S_{\sigma_j}$ is independent of $n$ (the $U = \mathbf{1}$ term contains $S_{\sigma_0} = S_{n-|\sigma|}$ which does depend on $n$).
\newline \newline \noindent
For a function $f: I_U \to \mathcal{P}$, let $c_f \in \mathbb{Z}_{\geq 0}$ be defined by restricting representations of symmetric groups:
\[
\Res{\prod_{i \in I_U} S_{\sigma_i}}{S_{|\vv{\mu}[n](U)|}}\left(\mathcal{S}^{\vv{\mu}[n](U)}\right) = \bigoplus_{f: I_U \to \mathcal{P}, |f(i)| = \sigma_i} \left(\boxtimes_{i \in I_U} \mathcal{S}^{f(i)} \right)^{\oplus c_f}
\]
After suitably applying Proposition \ref{Wprop2} to decompose such a restriction, we get the following.
\[
\Res{S_{|\vv{\mu}[n](U)|} \cap tH_{\vv{\nu}[n]}t^{-1}}{S_{|\vv{\mu}[n](U)|}}\left(U^{\boxtimes |\vv{\mu}[n](U)|} \otimes \mathcal{S}^{\vv{\mu}[n](U)}\right)
=
\bigoplus_{f: I_U \to \mathcal{P}, |f(i)| = \mu_i} \left(\boxtimes_{i \in I_U} \left(U^{\boxtimes |f(i)|} \otimes \mathcal{S}^{f(i)} \right)\right)^{\oplus c_f}
\]
This makes it clear that the only dependence on $n$ enters through the term corresponding to $\sigma_0 = n - |\sigma|$ in the $U = \mathbf{1}$ case (the cases for other $U$ do not involve $S_{n-|\sigma|}$ and hence do not depend on $n$). Let $k = \sum_{i \in I_{\mathbf{1}}\backslash \{0\}} \sigma_i$. In that case, we observe that restriction from $S_{|\vv{\mu}[n](\mathbf{1})|}$ to $(\prod_{i \in I_{\mathbf{1}} \backslash \{0\}} S_{\sigma_i}) \times  S_{n - |\sigma|}$ is the same as first restricting to $S_{k} \times S_{n - |\sigma|}$ and then restricting the first factor to $\prod_{i \in I_{\mathbf{1}} \backslash \{0\}} S_{\sigma_i}$ where the latter operation will be independent of $n$, similarly to the case $U \neq \mathbf{1}$. Thus it is enough to show that the operation of restricting to $S_{k} \times S_{n - |\sigma|}$ is stable in the sense described by the statement of the lemma.
\newline \newline \noindent
To understand the restriction, we fix an integer partition $\rho$. We must demonstrate the stability of the following expression (understood as a sum of terms of the form $\mathcal{S}^{\eta} \boxtimes \mathcal{S}^{(n - k - |\tau|, \tau)}$ for $\eta \vdash k$).
\[
\Res{S_{k} \times S_{n - |\mu|}}{S_{n-|\mu|+k}}(\mathcal{S}^{(n -k - |\rho|, \rho)})
\]
The restriction multiplicities are given by Littlewood-Richardson coefficients, and the stability condition is immediately implied by Proposition \ref{LRS}.
\end{proof}
\begin{remark}
The analogous stability statement for
\[
 \Res{H_{\vv{\mu}[n]} \cap tH_{\vv{\nu}[n]}t^{-1}}{tH_{\vv{\nu}[n]}t^{-1}}\left(t \boxtimes_{U \in I(\mathcal{C})} \left(U^{\boxtimes |\vv{\nu}[n](U)|} \otimes \mathcal{S}^{\vv{\nu}[n](U)} \right) \right)
\]
is proved similarly.
\end{remark}
\noindent
Finally, we are able to prove stability of the coefficients $c_{\vv{\mu}, \vv{\nu}}^{\vv{\lambda}}(n)$.
\begin{theorem} \label{StrCstStab}
For any choice of $\vv{\mu}, \vv{\nu}, \vv{\lambda}$, $\lim_{n \to \infty}c_{\vv{\mu}, \vv{\nu}}^{\vv{\lambda}}(n)$ exists and is a nonnegative integer.
\end{theorem}
\begin{proof}
We use Lemma \ref{PL}, again reducing to the case of a fixed double coset representative $t$ (of which there are finitely many, and they are stable with respect to $n$). We must demonstrate stability of
\begin{align*} 
\Ind{H_{\vv{\mu}[n]} \cap tH_{\vv{\nu}[n]}t^{-1}}{S_n}\left( 
 \Res{H_{\vv{\mu}[n]} \cap tH_{\vv{\nu}[n]}t^{-1}}{H_{\vv{\mu}[n]}}\left(\boxtimes_{U \in I(\mathcal{C})} \left(U^{\boxtimes |\vv{\mu}[n](U)|} \otimes \mathcal{S}^{\vv{\mu}[n](U)}\right)\right) \right. \\
 \otimes \left. \Res{H_{\vv{\mu}[n]} \cap tH_{\vv{\nu}[n]}t^{-1}}{tH_{\vv{\nu}[n]}t^{-1}}\left(t \boxtimes_{U \in I(\mathcal{C})} \left(U^{\boxtimes |\vv{\nu}[n](U)|} \otimes \mathcal{S}^{\vv{\nu}[n](U)} \right)\right) \right)
\end{align*}
The restrictions give a finite number of stable summands (by Lemma \ref{PL}), so it suffices to show that products of summands exhibit the relevant stabilisation property. We write:
\begin{eqnarray*}
& &\Ind{H_{\vv{\mu}[n]} \cap tH_{\vv{\nu}[n]}t^{-1}}{S_n}\left(
 \Res{H_{\vv{\mu}[n]} \cap tH_{\vv{\nu}[n]}t^{-1}}{H_{\vv{\mu}[n]}}\left(\boxtimes_{U \in I(\mathcal{C})} \left( U^{\boxtimes |\vv{\mu}[n](U)|} \otimes \mathcal{S}^{\vv{\mu}[n](U)}\right)\right) \right. \\
 & & \hspace{34.2mm}
 \otimes
\left. \Res{H_{\vv{\mu}[n]} \cap tH_{\vv{\nu}[n]}t^{-1}}{tH_{\vv{\nu}[n]}t^{-1}}\left(t \boxtimes_{U \in I(\mathcal{C})} \left(U^{\boxtimes |\vv{\nu}[n](U)|} \otimes \mathcal{S}^{\vv{\nu}[n](U)}\right)\right)\right) \\
&=&
\bigoplus
\Ind{S_\sigma \times S_{n - |\sigma|}}{S_n}\left( 
\left( \boxtimes_{i = 1}^{l} \left( U_i^{\boxtimes |\tau^{(i)}|} \otimes \mathcal{S}^{\tau^{(i)}}\right) \right) \boxtimes \left(\mathbf{1}^{\boxtimes (n - |\sigma|)} \otimes \mathcal{S}^{(n - |\tau^{(0)}| - |\sigma|, \tau^{(0)})}\right)
\right.
\\
& & \hspace{24.2mm} \left.
\otimes 
\left( \boxtimes_{i = 1}^{l} \left(V_i^{\boxtimes |\rho^{(i)}|} \otimes \mathcal{S}^{\rho^{(i)}}\right) \right) \boxtimes \left(\mathbf{1}^{\boxtimes (n - |\sigma|)} \otimes \mathcal{S}^{(n - |\rho^{(0)}| - |\sigma|, \rho^{(0)})}\right)
\right) \\
&=&
\bigoplus
\Ind{S_\sigma \times S_{n - |\sigma|}}{S_n}\left( 
\left( \boxtimes_{i = 1}^{l} \left((U_i\otimes V_i)^{|\tau^{(i)}|} \otimes \mathcal{S}^{\tau^{(i)}} \otimes \mathcal{S}^{\rho^{(i)}} \right) \right)
\right.
\\
& & \hspace{24.2mm} \left.
\boxtimes \left( \mathbf{1}^{\boxtimes(n - |\sigma|)} \otimes \mathcal{S}^{(n - |\tau^{(0)}| - |\sigma|, \tau^{(0)})} \otimes \mathcal{S}^{(n - |\rho^{(0)}| - |\sigma|, \rho^{(0)})}\right)
\right)
\end{eqnarray*}
Here the first equality used the statement of the preceding lemma (hence the implied sum is finite and independent of $n$); $\tau^{(i)}$ and $\rho^{(i)}$ are the partitions coming from the statement of the lemma. The second equality used Proposition \ref{Wprop1}. Each $U_i \otimes V_i$ decomposes into a linear combination of simple $[U]$ when we pass to the Grothendieck group. Proposition \ref{Wprop3} can be used to replace the summand with a sum of similar terms where the $U_i \otimes V_i$ are replaced with $[U]$ for some $U \in I(\mathcal{C})$ and the resulting quantity is independent of $n$. The result that the term $\mathbf{1}^{\boxtimes (n - |\sigma|)} \otimes \left(\mathcal{S}^{(n - |\tau^{(0)}| - |\sigma|, \tau^{(0)})} \otimes \mathcal{S}^{(n - |\rho^{(0)}| - |\sigma|, \rho^{(0)})}\right)$ admits a stable limit in terms of $\mathbf{1}^{\boxtimes (n -|\sigma|)} \otimes \mathcal{S}^{(n - |\lambda| - |\sigma|, \lambda)}$ is equivalent to the stability of Kronecker coefficients. Then, taking the exterior tensor product with a finite number of fixed $\mathbf{1}^{\boxtimes |\tau^{(i)}|} \otimes \mathcal{S}^{\tau^{(i)}}$ (coming from the finite terms in the product) and inducing to a larger symmetric group:
\[
\Ind{S_{n-|\sigma|} \times \prod_i S_{|\tau^{(i)}|}}{S_n} \left(
\left( \mathbf{1}^{\boxtimes (n - |\sigma|)} \otimes \mathcal{S}^{(n - |\lambda| - |\sigma|, \lambda)} \right) \boxtimes \left(\boxtimes_{i=1}^{l} \mathbf{1}^{\boxtimes |\tau^{(i)}|} \otimes \mathcal{S}^{\tau^{(i)}}\right) \right)
\]
also has a stable limit because the multiplicities are described by Littlewood-Richardson coefficients which we already know have suitable stability properties as per Proposition \ref{LRS}. We obtain a linear combination of $[R_{\vv{\lambda}[n]}]$ in the Grothendieck group. The coefficients are finite because they are limits of an eventually constant sequence of integers.
\end{proof}
\subsection{Definition and Basic Properties of the Limiting Grothendieck Ring}
We come to the definition of the main object of this paper.
\begin{definition}
Let $\mathcal{G}_{\infty}(\mathcal{C})$ be the $\mathbb{Q}$-vector space having basis $X_{\vv{\lambda}}$ for $\vv{\lambda} \in \mathcal{P}^\mathcal{C}$ and a multiplication defined by:
\[
X_{\vv{\mu}} X_{\vv{\nu}} = \sum_{\vv{\lambda} \in \mathcal{P}^\mathcal{C}}
\left(
\lim_{n \to \infty} c_{\vv{\mu}, \vv{\nu}}^{\vv{\lambda}}(n)
\right) X_{\vv{\lambda}}
\]
We will show that this is multiplication is associative and unital, making $\mathcal{G}_\infty(\mathcal{C})$ into an associative $\mathbb{Q}$-algebra.
\end{definition}
\begin{remark} \label{integral_rmk}
We could have defined $\mathcal{G}_{\infty}(\mathcal{C})$ over $\mathbb{Z}$ instead of over $\mathbb{Q}$, but then certain elements of interest to us would no longer lie in the algebra.
\end{remark}
\noindent
We also introduce a collection of elements that will be important.
\begin{definition}
A \textit{basic hook} is an element $\vv{\lambda} \in \mathcal{P}^\mathcal{C}$ such that $\vv{\lambda}(U) = (1^n)$ for some $U \in I(\mathcal{C})$, and $\vv{\lambda}(V)$ is the empty partition for all $V$ different from $U$. By abuse of terminology we also refer to $X_{\vv{\lambda}} \in \mathcal{G}_{\infty}(\mathcal{C})$ as a basic hook whenever the indexing multipartition $\vv{\lambda}$ is a basic hook;  we also denote $X_{\vv{\lambda}}$ as $\overline{e}_n(U)$.
\end{definition}
\begin{theorem}\label{delconst}
Asymptotically as $t \to \infty$, the structure constants of the images of indecomposable objects of the Deligne category $S_t(\mathcal{C})$ in the relevant Grothendieck group agree with the structure constants of the $X_{\vv{\lambda}}$ in $\mathcal{G}_\infty(\mathcal{C})$.
\end{theorem}
\begin{proof}
The wreath product Deligne categories $S_t(\mathcal{C})$ admit tensor functors to $\Wreath{n}$ for $n \in \mathbb{Z}_{\geq 0}$, and their behaviour is discussed in Theorem 5.6 \cite{Mori}, and restated in our setting in Theorem 3.1 of \cite{Nate}. The object indexed by a multipartition $\vv{\lambda}$ is mapped to the irreducible object of $\Wreath{n}$ indexed by $\vv{\lambda}[n]$, if $\vv{\lambda}[n]$ is a well defined multipartition (i.e. $n - |\vv{\lambda}| \geq \vv{\lambda}(\mathbf{1})_1$), and otherwise it is zero. In \cite{Mori}, Theorem 4.13 demonstrates that for fixed objects of $S_t(\mathcal{C})$ and sufficiently large $t$, spaces of homomorphisms in $S_t(\mathcal{C})$ can be computed by using the tensor functor to pass to the wreath product categories $\Wreath{n}$. The tensor product multiplicities are determined by the homomorphism spaces in the following way. An object $M$ of $S_t(\mathcal{C})$ is determined by the information $\hom_{S_t(\mathcal{C})}(N, M)$ for all objects $N$ (by the Yoneda lemma). This information can be recovered by passing to sufficiently large wreath product categories $\Wreath{n}$, and in particular we can choose $M$ to be a tensor product of two objects. In the original setting of $\underline{\text{Rep}}(S_t)$, the result was proved by Deligne in \cite{deligne}.
\end{proof}
\noindent
We have a few preliminary facts about the algebra $\mathcal{G}_{\infty}(\mathcal{C})$.
\begin{theorem} \label{prelimFacts}
The algebra $\mathcal{G}_{\infty}(\mathcal{C})$ is a unital associative algebra satisfying the following:
\begin{enumerate}
\item $\mathcal{G}_{\infty}(\mathcal{C})$ is commutative if and only if $\mathcal{G}(\mathcal{C})$ is commutative.
\item $\mathcal{G}_{\infty}(\mathcal{C})$ is generated by the basic hooks $\overline{e}_n(U)$, where $n \geq 1$ and $U \in I(\mathcal{C})$.
\item There is a filtration $\mathcal{G}_{\infty}(\mathcal{C}) = \cup_{n \in \mathbb{N}}\mathcal{F}_n$, where $\mathcal{F}_n$ is spanned by $X_{\vv{\lambda}}$ with $|\vv{\lambda}| \leq n$.
\item The associated graded algebra of $\mathcal{G}_{\infty}(\mathcal{C})$ with respect to this filtration is isomorphic to $\bigotimes_{U \in I(\mathcal{C})} \Lambda_\mathbb{Q}^{(U)}$, where $\Lambda_{\mathbb{Q}}^{(U)}$ is the ring of symmetric functions with coefficients in $\mathbb{Q}$. If we write $f^{(U)}$ to indicate that the symmetric function $f$ is considered as an element of $\Lambda_{\mathbb{Q}}^{(U)}$, then the image of $[R_{\vv{\lambda}}]$ is $\prod_{U \in I(\mathcal{C})} s_{\vv{\lambda}(U)}^{(U)}$.
\end{enumerate}
\end{theorem}
\begin{proof}
Firstly, the multiplication in $\mathcal{G}_{\infty}(\mathcal{C})$ is seen to be associative by considering $[R_{\vv{\lambda}[n]}][R_{\vv{\mu}[n]}][R_{\vv{\nu}[n]}]$ in the associative algebra $\mathcal{G}(\Wreath{n})$. For $n$ sufficiently large, the coefficient of $[R_{\vv{\rho}[n]}]$ in that element becomes equal to the coefficient of $X_{\vv{\rho}}$ in $X_{\vv{\lambda}}X_{\vv{\mu}}X_{\vv{\nu}}$, regardless of how the latter product is parenthesised. The basis element corresponding to the empty partition is the unit element.
\newline \newline \noindent
The commutativity or non-commutativity of multiplication can be seen from the proof of Theorem \ref{StrCstStab}, where (up to conjugation by the double coset representative $t$) the only change between $[R_{\vv{\mu}}][R_{\vv{\nu}}]$ and $[R_{\vv{\nu}}][R_{\vv{\mu}}]$ is the product $U_i \otimes V_i$ (versus $V_i \otimes U_i$) of objects of $\mathcal{C}$. Proposition \ref{Wprop3} was used to write the result in terms of $[U]$ for $U \in I(\mathcal{C})$, and equal results are obtained for $U_i \otimes V_i$ and $V_i \otimes U_i$ if and only if $[U_i \otimes V_i] = [V_i \otimes U_i]$. This holds for all possible choices of $U_i$ and $V_i$ if and only if $\mathcal{G}(\mathcal{C})$ is commutative.
\newline \newline \noindent
The filtration is essentially the same as the $|\lambda|$-filtration defined in Definition 2.7 of \cite{Nate}. 
In particular, the associated graded algebra (with basis induced from $X_{\vv{\lambda}}$) has structure constants equal to those of the ring of symmetric functions with the Schur function basis. In particular,  the basic hooks $\overline{e}_n(U)$ correspond to elementary symmetric functions $e_n^{(U)}$ in $\Lambda_{\mathbb{Q}}^{(U)}$. This means that the basic hooks generate the associated graded algebra, and hence they generate $\mathcal{G}_{\infty}(\mathcal{C})$.
\end{proof}
\begin{example}
In the case where $\mathcal{C} = kG-mod$ for a finite group $G$, $\Wreath{n}$ is equivalent to the category of finite-dimensional representations for the wreath product $G^{n} \rtimes S_n$. In this case, $\mathcal{G}_{\infty}(\mathcal{C})$ is commutative, it follows that $\mathcal{G}_{\infty}(\mathcal{C})$ is isomorphic to a free polynomial algebra in the basic hooks; see Corollary 2.9 of \cite{Nate}. In our setting, $\mathcal{G}(\mathcal{C})$ may not be commutative, in which case $\mathcal{G}_{\infty}(\mathcal{C})$ cannot possibly be a free polynomial algebra. Nevertheless, we will give a description of the algebra structure of the ring in terms of basic hooks, and also give generating functions describing how the $X_{\lambda}$ are expressed in terms of basic hooks.
\end{example}

\section{The Ring Structure of $\mathcal{G}_{\infty}(\mathcal{C})$}
\subsection{The Elements $T_n(M)$} \noindent
We use the following construction to relate $\Wreath{n}$ with $\mathcal{G}_\infty(\mathcal{C})$; informally, we take $n \to \infty$.
\begin{definition}
Suppose that $H$ is a subgroup of $S_n$ and $M$ is a module over $\hopf \wr H$.
In $\mathcal{G}(\Wreath{n+m})$ We may write
\begin{equation*}
[\Ind{H \times S_m}{S_{m+n}}\left(M \boxtimes \left(\mathbf{1}^{\boxtimes m} \otimes \mathbf{1}_{S_m}\right) \right)] = \sum_{\vv{\mu} \in \mathcal{P}^{\mathcal{C}}} c_{\vv{\mu}}(M, m) [R_{\vv{\mu}[n+m]}]
\end{equation*}
Here $\mathbf{1}_{S_m}$ is the trivial representation of $S_m$ and for any fixed $n$ we only sum over $\vv{\mu}$ such that $\vv{\mu}[n+m]$ is defined. By transitivity of induction, we may replace $M$ and $H$ with $\Ind{H}{S_n}(M)$ and $S_n$ respectively. In this case, if $\Ind{H}{S_{n}}(M)$ is a simple object, it is induced from an object of the form (where $\vv{\rho} \in \mathcal{P}^\mathcal{C}$)
\[
M = \left(\mathbf{1}^{\boxtimes |\vv{\rho}[n](\mathbf{1})|} \otimes \mathcal{S}^{\vv{\rho}[n](\mathbf{1})}\right) \boxtimes \left(\boxtimes_{U \neq \mathbf{1}} \left(U^{\boxtimes |\vv{\rho}[n](U)|} \otimes \mathcal{S}^{\vv{\rho}[n](U)} \right) \right)
\]
Substituting this into the equation defining $c_{\vv{\mu}}(M, m)$, we see that the stability of Littlewood-Richardson coefficients (Proposition \ref{LRS}) implies that there is a nonzero contribution to only finitely many $c_{\mu}(M,m)$, and the contribution becomes constant for $m$ sufficiently large. We define
\[
\lim_{m \to \infty} \Ind{H \times S_m}{S_{m+n}}\left(M \boxtimes \left(\mathbf{1}^{\boxtimes m} \otimes \mathbf{1}_{S_m}\right)\right) = \sum_{\mu \in \mathcal{P}_{\mathcal{C}}} \left(
\lim_{m \to \infty} c_\mu(M, m)
\right) X_{\mu}
\]
By the linearity of induction, we may extend this definition to allow $M$ to be not necessarily simple, or indeed a formal difference of objects (when working with Grothendieck groups). Note that this construction only depends on the class of $M$ in the Grothendieck group because induction is an exact functor. So, we may write $\lim_{m \to \infty} [\Ind{H \times S_m}{S_{m+n}}] (-)$ when the argument is an element of a Grothendieck group (rather than an object of a category), and the operation is still well defined.
\end{definition}
\noindent
We now define a generating set of $\mathcal{G}_{\infty}(\mathcal{C})$ with favourable multiplicative properties.
\begin{definition}
For an object $M$ of $\mathcal{C}$ and $n \in \mathbb{Z}_{>0}$, we define the following elements of $\mathbb{Q} \otimes_{\mathbb{Z} }\mathcal{G}(\Wreath{n})$:
\[
T_{n}^{f}(M) = \frac{1}{n} \sum_{\lambda \vdash n} \chi_{(n)}^{\lambda} [M^{\boxtimes |\lambda|} \otimes \mathcal{S}^{\lambda}]
\]
We construct an analogous element of $\mathcal{G}_{\infty}(\mathcal{C})$ as follows.
\[
T_n(M) = \frac{1}{n}\sum_{\lambda \vdash n}\chi_{(n)}^{\lambda}\lim_{m \to \infty}\Ind{S_n \times S_m}{S_{n+m}}\left(\left(M^{\boxtimes |\lambda|} \otimes \mathcal{S}^{\lambda} \right) \boxtimes \left(\mathbf{1}^{\boxtimes m} \otimes \mathbf{1}_{S_m}\right)\right)
\]
\noindent
We write $[\Ind{S_{n}^{k}\times S_m}{S_{nk+m}}]$ for function on Grothendieck groups $\mathcal{G}(\Wreath{n})^{\boxtimes k} \boxtimes \mathcal{G}(\Wreath{m}) \to \mathcal{G}(\Wreath{nk+m})$ induced by the induction functor. Now we let
\[
T_n(M_1, M_2, \cdots, M_k) = \lim_{m \to \infty} [\Ind{S_{n}^{k}\times S_m}{S_{nk+m}}]\left(T_n^{f}(M_1)\boxtimes T_n^{f}(M_2) \boxtimes \cdots \boxtimes T_n^{f}(M_k) \boxtimes [\left(\mathbf{1}^{\boxtimes m} \otimes \mathbf{1}_{S_m}\right)]\right)
\]
As before, $T_n(M_1, M_2, \cdots, M_k)$ only depends on the class of the $M_i$ in the corresponding Grothendieck groups. \end{definition}

\begin{remark}
The character orthogonality relation
\[
\frac{1}{n} \sum_{\lambda \vdash n} \chi_{(n)}^{\lambda} \chi^{\lambda}_{\nu} = \delta_{\nu, (n)}
\]
suggests that one can think of $T_n(U)$ as a generalisation of the indicator function of cycle type $(n)$ in a copy of the class functions on $S_n$ associated to $U \in I(\mathcal{C})$.  
This is based on the fact that the virtual character associated to $\frac{1}{n} \sum_{\lambda \vdash n} \chi_{(n)}^{\lambda}\mathcal{S}^{\lambda}$ is the indicator function of $n$-cycles on $S_n$. 
\end{remark}

\begin{proposition} \label{indepProp}
We have the following properties of the $T_n(U)$, for $U\in I(\mathcal{C})$:
\begin{enumerate}
\item The image of $T_n(U)$ in the associated graded algebra of $\mathcal{G}_{\infty}(\mathcal{C})$, which we identify with $\bigotimes_U \Lambda_\mathbb{Q}^{(U)}$, is $p_n^{(U)}/n$. That is, the $n$-th power sum symmetric function in $\Lambda_{\mathbb{Q}}^{(U)}$, divided by $n$.
\item Fix a total order on $\mathbb{Z}_{>0} \times I(\mathcal{C})$. Consider the monomials in $T_{n}(U)$ for $(n, U) \in \mathbb{Z}_{>0} \times I(\mathcal{C})$ where the factors occur in order consistent with the total order (``PBW monomials''). These monomials are linearly independent.
\item The $T_n(U)$ generate $\mathcal{G}_{\infty}(\mathcal{C})$.
\item $T_n(U)$ lies in the $n$-th filtered component of $\mathcal{G}_{\infty}(\mathcal{C})$.
\end{enumerate}
\end{proposition}
\begin{proof}
The first claim follows from the fact that the virtual character associated to $\frac{1}{n} \sum_{\lambda \vdash n} \chi_{(n)}^{\lambda}\mathcal{S}^{\lambda}$ is the indicator function of $n$-cycles on $S_n$. The second follows from the fact that the $p_n^{(U)}$ are algebraically independent in the associated graded algebra. Since the $p_n^{(U)}$ generate the associated graded algebra, the third claim follows. The final claim is immediate if $U$ is different from $\mathbf{1}$, for then $T_n(U)$ becomes a linear combination of $X_{\vv{\lambda}}$ with $|\vv{\lambda}| = n$. The $U = \mathbf{1}$ case follows from the Pieri rule (see also Remark \ref{adjointRmk}) which describes certain Littlewood-Richardson coefficients; we wish to decompose $\Ind{S_n \times S_m}{S_{n+m}} (\mathcal{S}^{\lambda} \boxtimes \mathbf{1}_{S_m})$ into $\mathcal{S}^{\mu[n+m]}$, with $|\mu| \leq n$. The $\mu$ that appear are obtained from $\lambda$ by adding $m$ boxes, no two in the same column, and then removing the top row. Removing the top row removes one box from each column, so any $\mu$ obtained this way satisfies $|\mu| \leq |\lambda| = n$.
\end{proof}
\noindent
The following lemma will underpin much of what follows. Note that restriction is an exact functor, so it descends to a function between Grothendieck groups.
\begin{lemma}\label{resLemma}
Any restriction of $T_n^{f}(U)$ to (the Grothendieck group of) a proper Young subgroup $S_\lambda$ of $S_n$ is zero. 
\end{lemma}
\begin{proof}
We use Proposition \ref{Wprop2}, part 2. It now suffices to understand how the indicator function of $n$-cycles restricts to $S_\lambda$. The result now follows from the fact that the only Young subgroup of $S_n$ containing an $n$-cycle is all of $S_n$.
\end{proof}
\noindent
In order to understand the algebra structure of $\mathcal{G}_{\infty}(\mathcal{C})$, we determine the commutator of two elements of the form $T_n(U)$ (recall that such elements generate the algebra).
\begin{theorem} \label{commThm}
Let $U_1, U_2 \in I(\mathcal{C})$. If $n \neq m$, the commutator of $T_n(U_1)$ and $T_m(U_2)$ vanishes: we have $[T_n(U_1), T_m(U_2)] = 0$.
 If $N_{U_1, U_2}^{U_3}$ is the structure tensor of the Grothendieck ring (so that $[U_1][U_2] = \sum_{U_3} N_{U_1, U_2}^{U_3}[U_3]$), then we have:
\[
[T_{n}(U_1), T_{n}(U_2)] = \sum_{U_3} (N_{U_1, U_2}^{U_3} - N_{U_2, U_1}^{U_3}) T_{n}(U_3)
\]
\end{theorem}
\begin{proof}
To calculate the commutator $[T_n(U_1), T_m(U_2)]$, we calculate the analogous quantity in the ring $\mathbb{Q} \otimes_{\mathbb{Z}} \mathcal{G}(\Wreath{m+n+k})$ for $k$ sufficiently large.
As $T_n(U_1), T_m(U_2)$ are defined as the image of a linear combination of induced objects in the Grothendieck group, we may apply the Mackey theory formalism to calculate $T_n(U_1) T_m(U_2)$. 
\newline \newline \noindent
Firstly, the number of $(S_n \times S_{m+k}, S_{m} \times S_{n+k})$-double cosets in $S_{m+n+k}$ is $\min(m,n)+1$. This can be seen from calculating the possible $C_{i,j}(\sigma)$ that can arise. Both $i$ and $j$ may take two different values. Therefore $C_{1,1}(\sigma)$ determines all other $C_{i,j}(\sigma)$ via identities such as $C_{i,1}(\sigma) + C_{i,2}(\sigma)= |B_1| = m$ (similarly $C_{1,j}(\sigma)$ and $C_{2,j}(\sigma)$ determine each other). So, double cosets are determined by a single invariant $C_{1,1}(\sigma)$, namely, the number of elements of $\{1,2,\ldots, m\}$ that are mapped to the set $\{1,2,\ldots, n\}$. Clearly $C_{1,1}(\sigma)$ can take any of the values $0, 1, \ldots, \min(m,n)$.
\newline \newline \noindent
Consider a double coset representative $\sigma$ (interpreted as a bijection from the set $\{1,2,\ldots, n+m+k\}$ to itself) such that $f(\{1,2, \ldots, n\}) \neq \{1,2, \ldots, m\}$ and $f(\{1,2,\ldots, n\}) \cap \{ 1,2,\ldots, m\} \neq \varnothing$. In the Mackey theoretic computation the summand coming from a twist by $\sigma$ will involve restricting to $(S_n \times S_{m+k}) \cap \sigma (S_{m} \times S_{n+k})\sigma^{-1}$, which will not contain the entirety of $S_n$ (considered as a subgroup of $S_n \times S_{m+k}$):
\begin{align*} 
\Ind{(S_n \times S_{m+k}) \cap \sigma (S_{m} \times S_{n+k})\sigma^{-1}}{S_{n+m+k}}\left( 
 \Res{(S_n \times S_{m+k}) \cap \sigma (S_{m} \times S_{n+k})\sigma^{-1}}{S_n \times S_{m+k}}
 \left(
T_n^f(U_1) \boxtimes \left( \mathbf{1}^{\boxtimes (m+k)} \otimes \mathbf{1}_{S_{m+k}}\right)
 \right)
 \right. \\
 \otimes \left. \Res{(S_n \times S_{m+k}) \cap \sigma (S_{m} \times S_{n+k})\sigma^{-1}}{\sigma (S_{m} \times S_{n+k})\sigma^{-1}}\left(
 \sigma \left(T_m^f(U_2) \boxtimes  \left( \mathbf{1}^{\boxtimes (n+k)} \otimes \mathbf{1}_{S_{n+k}}\right)\right)
\right) \right)
\end{align*}
In particular, the calculation involves restricting $T_n(U)$ to a proper Young subgroup of $S_n$, giving zero by Lemma \ref{resLemma}.
There are two cases that need to be considered; $C_{1,1}(\sigma) = 0$ and $C_{1,1}(\sigma) = n = m$. The first case gives rise to the following term:
\[
\Ind{S_n \times S_m \times S_k}{S_{n+m+k}}(T_n^f(U_1) \boxtimes T_m^f(U_2) \boxtimes [\mathbf{1}^{\boxtimes k} \otimes \mathbf{1}_{S_k}])
\]
Since $S_n \times S_m \times S_k$ and $S_m \times S_n \times S_k$ are conjugate subgroups of $S_{m+n+k}$, it follows that if $T_n^f(U_1)$ and $T_m^f(U_2)$ were interchanged, the contribution would be the same, in particular, the contribution of the term associated to this double coset is cancelled out in the commutator by the corresponding term in $T_m(U_2)T_n(U_1)$. In particular, if $n \neq m$, $[T_n(U_1), T_m(U_2)] = 0$.
\newline \newline \noindent
If $n = m$, then there we consider the contribution from the double coset representative which identifies the symmetric group factors associated to $S_n$ and $S_m$ in the respective Young subgroups. Working in $\mathcal{G}(\Wreath{n})$ and applying Proposition \ref{Wprop1} we find:
\[
T_n^f(U_1) T_n^f(U_2) = T_n^f (U_1 \otimes U_2)
\]
Here we used the fact that the indicator function of $n$-cycles (considered as a class function on $kS_n$) is an idempotent for the tensor product. We may use Proposition \ref{Wprop3} to express $T_n(U_1 \otimes U_2)$ in terms of $T_n(U)$ for $U \in I(\mathcal{C})$. We use the equation in Proposition \ref{Wprop3} multiplied by $T_n^f(\mathbf{1})$ (take $N$ to be a subobject of $M$ in $\mathcal{C}$):
\begin{eqnarray*}
T_n^{f}(M) &=& \sum_{\lambda \vdash n} \frac{\chi_{(n)}^{\lambda}}{n}[M^{\boxtimes n} \otimes \mathcal{S}^\lambda] \\
&=& [M^{\boxtimes n} \otimes \mathbf{1}_{S_n}] \sum_{\lambda \vdash n} \frac{\chi_{(n)}^{\lambda}}{n}[\mathbf{1}^{\boxtimes n} \otimes \mathcal{S}^\lambda] \\
&=& [M^{\boxtimes n} \otimes \mathbf{1}_{S_n}] T_n^{f}(\mathbf{1}) \\
&=& \sum_{r=0}^{n} [\Ind{S_r \times S_{n-r}}{S_n}\left((N^{\boxtimes r} \otimes \mathbf{1}_{S_r}) \boxtimes ((M/N)^{\boxtimes (n-r)} \otimes \mathbf{1}_{S_{n-r}})\right)] T_n^{f}(\mathbf{1})
\end{eqnarray*}
Now we use Lemma \ref{ind_res_tensor}, giving:
\[
\sum_{r=0}^{n} [\Ind{S_r \times S_{n-r}}{S_n}\left(\left((N^{\boxtimes r} \otimes \mathbf{1}_{S_r}) \boxtimes ((M/N)^{\boxtimes (n-r)} \otimes \mathbf{1}_{S_{n-r}})\right) \otimes \Res{S_r \times S_{n-r}}{S_n}\left( \sum_{\lambda \vdash n} \frac{\chi_{(n)}^{\lambda}}{n} (\mathbf{1}^{\boxtimes n} \otimes \mathcal{S}^\lambda) \right) \right)] 
\]
Note that the argument in Lemma \ref{resLemma} implies that all terms except for $r=0,n$ vanish (they involve the restriction of the indicator function of $n$-cycles to a proper Young subgroup of $S_n$). We get $T_n^{f}(M) = T_n^{f}(N) + T_n^{f}(M/N)$, and this immediately implies $T_n(M) = T_n(N) + T_n(M/N)$. Iterating this, we get one term for each composition factor of $U_1 \otimes U_2$. If $N_{U_1, U_2}^{U_3}$ is the structure tensor of the Grothendieck ring, then we have:
\[
[T_{n}(U_1), T_{n}(U_2)] = \sum_{U_3} (N_{U_1, U_2}^{U_3} - N_{U_2, U_1}^{U_3}) T_{n}(U_3)
\]
\end{proof}
\begin{remark} \label{lie_hom_rmk}
We may summarise these results by saying that the map from the Grothendieck ring of $\mathcal{C}$ (with coefficients in $\mathbb{Q}$) to the $\mathbb{Q}$-span of the $T_n(U)$ defined by $[U] \mapsto T_n(U)$ is a homomorphism of Lie algebras. The fact that $T_n(M) = T_n(N) + T_n(M/N)$ shows linearity, and we have just shown that it preserves the Lie bracket.
\end{remark}
\subsection{Structure of the Limiting Grothendieck Ring}
We are now able to give a presentation of $\mathcal{G}_{\infty}(\mathcal{C})$. Recall that if $A_i$ is an infinite family of unital algebras over $k$, then the the tensor product $\bigotimes_i A_i$ is spanned by pure tensors $a_1 \otimes a_2 \otimes \cdots $ ($a_i \in A_i$) whose factors are the unit elements in their respective algebras for all but finitely many $i$.
\begin{theorem}\label{strThm}
The $\mathbb{Q}$-algebra structure on $\mathcal{G}_{\infty}(\mathcal{C})$ is as follows:
\[
\mathcal{G}_{\infty}(\mathcal{C}) = \bigotimes_{i=1}^{\infty} U(\mathcal{G}(\mathcal{C})_i)
\]
Here, $U(\mathcal{G}(\mathcal{C})_i)$ is the universal enveloping algebra of the span of $T_{i}(U)$ for $U \in I(\mathcal{C})$ (this Lie algebra is contained within the $i$-th filtered component of $\mathcal{G}_{\infty}(\mathcal{C})$).
\end{theorem}
\begin{proof}
We have a map that takes $[U] \in U(\mathcal{G}(\mathcal{C})_n)$ to $T_n(U)$. It is a homomorphism by Remark \ref{lie_hom_rmk}. It is a bijection by Proposition \ref{indepProp}; it is surjective because the $T_n(U)$ generate $\mathcal{G}_\infty(\mathcal{C})$ and it is injective because the map is an isomorphism upon taking associated graded algebras.
\end{proof}
\begin{remark}
The previous theorem generalises the result that $\underline{\mbox{Rep}}(\hopf \wr S_t)$ is the free polynomial algebra generated by basic hooks when $\hopf$ is cocommutative, as the universal enveloping algebra of an abelian Lie algebra is a free polynomial algebra.
\end{remark}

\section{Partition Combinatorics}
\noindent
We now focus on finding an expression for $X_{\vv{\lambda}}$ in terms of the $T_n(U)$.
\subsection{Irreducibles in Terms of $T_n(U)$}
\begin{lemma} \label{class_ind}
If $D_{(n)}$ denotes the class function on $S_n$ which is the indicator function of $n$-cycles, then  the class function $\Ind{S_n^{m}}{S_{mn}}(D_{(n)}^{\otimes m})$ is $m!$ times the indicator function of elements of cycle type $(n^m)$.
\end{lemma}
\begin{proof}
This must be some multiple of the indicator function of elements of cycle type $(n^m)$. The multiplicity can be found using the Frobenius character formula for induced representations. Take $H$ a subgroup of $G$, $G/H$ a collection of left coset representatives, then the induction of a character $\chi$ from $H$ to $G$ is given by:
\[
\Ind{H}{G}(\chi) (x) = \sum_{g \in G/H, g^{-1}xg \in H} \chi(g^{-1}xg)
\]
which demonstrates that the multiplicity is in fact the index of $S_n^m$ in its normaliser in $S_{nm}$. The normaliser is the wreath product $S_m \ltimes S_n^m$, hence the index is $m!$.
\end{proof}
\noindent
For now we will fix $m$, and consider relations between the $T_m(U)$.
\begin{definition}
If $\lambda$ is a partition of $n$, let
\[
T_{m,\lambda}(a_1, a_2, \ldots, a_n) = T_m(a_1 a_2 \cdots a_{\lambda_1})T_m(a_{\lambda_1 + 1} a_{\lambda_1 + 2} \cdots a_{\lambda_1 + \lambda_2}) \cdots T_m(a_{n - \lambda_{l(\lambda)}+1} a_{n - \lambda_{l(\lambda)}+1} \cdots a_{n})
\]
\end{definition}
\begin{proposition} \label{tdecomp}
We have the following identity in $\mathcal{G}_{\infty}(\mathcal{C})$.
\[
T_m(a_1, a_2, \ldots, a_n) = \sum_{\sigma \in S_n} \sum_{\lambda \vdash n} \frac{\varepsilon_\lambda}{z_\lambda} T_{m,\lambda}(a_{\sigma(1)}, a_{\sigma(2)}, \ldots, a_{\sigma(n)})
\]
\end{proposition}
\noindent
Before the proof, we note a useful corollary.
\begin{corollary} \label{usefulCor}
In the case where $a_i = a$ for all $1 \leq i \leq n$, writing $\lambda = (1^{m_1}2^{m_2}\cdots n^{m_n})$ we obtain:
\[
T_m(a, a, \ldots, a) = \sum_{\lambda \vdash n} \frac{\varepsilon_\lambda n!}{z_\lambda} T_m(a)^{m_1}T_m(a^2)^{m_2} \cdots T_m(a^n)^{m_n}
\]
\end{corollary}
\begin{proof}
We use Mackey theory to calculate $T_m(b)T_m(a_1, a_2, \ldots, a_n)$. This can be understood by taking the tensor product of 
\[
\Ind{S_m \times S_{(n-1)m +k}}{S_{nm+k}} (T_m^{f}(b) \boxtimes (\mathbf{1}^{\boxtimes ((n-1)m+k)} \otimes \mathbf{1}_{S_{(n-1)m+k}}))
\]
and
\[
\Ind{S_m^n \times S_k}{S_{mn+k}}(T_m^f(a_1) \boxtimes T_n^f(a_2) \boxtimes \cdots \boxtimes T_n^f(a_n) \boxtimes (\mathbf{1}^{\boxtimes k} \otimes \mathbf{1}_{S_k})
\]
The first step is to understand double-coset representatives. The minimal length $(S_m \times S_{(n-1)m+k}, S_m^n\times S_k)$-double coset representatives either map the elements of $\{1,2, \ldots, m\}$ to a contiguous block of $m$ elements permuted by a single component group in $S_m^n\times S_k$, or the elements are split between such component groups. In the latter case, the corresponding terms (in the Mackey theory computation) will involve a nontrivial restriction of a $T_m^f$ to a Young subgroup, as in the proof of Theorem \ref{commThm}, a nontrivial restriction is zero. Thus, we consider the ways to pick a copy of $S_m$ as one of the $n$ given ones, or one contained in $S_k$. Analogously to Theorem \ref{commThm}, each gives rise to a term where the arguments multiply:
\[
T_m(b)T_m(a_1, a_2, \ldots, a_n) = T_m(b, a_1, a_2, \ldots, a_n) + \sum_{i=1}^n T_m(a_1, \ldots a_{i-1}, ba_i, a_{i+1}, \ldots, a_n)
\]
Using this equation, we may decompose the claimed expression for $T_m(a_1, a_2, \ldots, a_n)$ into a linear combination of $T_m(b_1, b_2, \ldots, b_m)$, where each $b_j$ is a product of $a_i$s. We count the coefficient of a term of the following form in the expression on the right hand side of the claimed equality:
\[
T_m(a_{r_{1,1}}a_{r_{1,2}}\cdots a_{r_{1, q_1}}, \ldots, a_{r_{k,1}}a_{r_{k,2}}\cdots a_{r_{k, q_k}})
\]
Here, the $r_{i,j}$ for $1 \leq i \leq k$, $1 \leq j \leq q_i$ are exactly the numbers $1$ through $n$ in some order.
\newline \newline \noindent
The argument $a_{r_{1,1}}a_{r_{1,2}}\cdots a_{r_{1, q_1}}$ must arise from a product of terms such as the following:
\[
T_m(a_{r_{1,1}}a_{r_{1,2}}\cdots a_{r_{1, \mu_1}})
T_m(a_{r_{1,\mu_1+1}}a_{r_{1,\mu_1+2}}\cdots a_{r_{1, \mu_1+\mu_2}})
\cdots
T_m(a_{r_{1,q_1 - \mu_l+1}}a_{r_{1,2}}\cdots a_{r_{1, q_k}})
\]
Since all terms $T_m(x)$ in the definition of $T_{m,\lambda}$ occur in non-increasing order of the number of factors in the argument $x$, we obtain a partition of $q_1$, $\mu^{(1)} = (\mu_1, \mu_2, \ldots, \mu_l)$, associated to the sequence $r_{1, j}$ which describes the factors $T_m(x)$ contributing to that term. A similar description holds for the other $r_{i,j}$ for other values of $i$. We obtain a description of all contributions; note that the $\lambda$ appearing in the sum will be the union of all $\mu^{(i)}$, and that if multiple $\mu^{(i)}$ have parts of some size $s$, then there is no restriction on the ordering of the corresponding $T_m(b_1 b_2\cdots b_s)$ terms within $T_{m,\lambda}$ (each possible ordering has an equal contribution). The coefficient of $T_m(a_{r_{1,1}}a_{r_{1,2}}\cdots a_{r_{1, q_1}}, \ldots, a_{r_{k,1}}a_{r_{k,2}}\cdots a_{r_{k, q_k}})$ is:
\[
\sum_{\mu^{(1)} \vdash q_1}
\sum_{\mu^{(2)} \vdash q_2}
\cdots
\sum_{\mu^{(k)} \vdash q_k}
\frac{\varepsilon_{{\cup_i \mu^{(i)}}}}{z_{\cup_i \mu^{(i)}}}
\prod_{j=1}^{n} \frac{\left(\sum_{i=1}^{k} m_j(\mu^{(i)})\right)!}{\prod_{i=1}^{k} m_j(\mu^{(i)})!}
\]
Here the multinomial coefficient arose because different orderings of factors can give rise to the same term. Using the fact that $\varepsilon_{\mu \cup \nu} = \varepsilon_\mu \varepsilon_\nu$ and the definition of $z_\mu$, our equation becomes:
\[
\sum_{\mu^{(1)} \vdash q_1}
\sum_{\mu^{(2)} \vdash q_2}
\cdots
\sum_{\mu^{(k)} \vdash q_k}
\prod_{i=1}^{k} \frac{\varepsilon_{{\mu^{(i)}}}}{z_{\mu^{(i)}}}
=
\prod_{i=1}^{k} \left(\sum_{\mu^{(i)} \vdash q_i} \frac{\varepsilon_{\mu^{(i)}} }{z_{\mu^{(i)}}}\right)
\]
The expansions for elementary and complete symmetric functions in terms of power sum symmetric functions show that we have:
\[
\delta_{n, 1} = \langle s_{(n)}, s_{(1^n)} \rangle=\langle h_n, e_n \rangle = \sum_{\lambda \vdash n} \langle \frac{p_\lambda}{z_\lambda}, \frac{\varepsilon_{\lambda}p_\lambda}{z_\lambda} \rangle=  \sum_{\lambda \vdash n} \frac{\varepsilon_{\lambda} }{z_{\lambda}}
\]
This means that the expression of interest vanishes unless each $q_i = 1$. In that case the constant is 1, and we simply obtain $T(a_1, a_2, \ldots, a_n)$ as claimed.
\end{proof}

\begin{proposition} \label{semiirred}
We may relate $U^{\boxtimes |\lambda|} \otimes \mathcal{S}^{\lambda}$ to the $T_i(U)$ as follows.
\begin{eqnarray*}
& &\lim_{m \to \infty} \Ind{S_{|\lambda|} \times S_m}{S_{|\lambda|+m}} \left(\left( U^{\boxtimes |\lambda|}\otimes \mathcal{S}^{\lambda}\right) \boxtimes \left(\mathbf{1}^{\boxtimes m} \otimes \mathbf{1}_{S_m}\right)\right) \\
&=& \sum_{\mu \vdash |\lambda|} \chi_{\mu}^{\lambda} \frac{T_1(\overbrace{U, U, \ldots, U}^{m_1(\mu)})}{m_1(\mu)!} \frac{T_2(\overbrace{U, U, \ldots, U}^{m_2(\mu)})}{m_2(\mu)!} \cdots \frac{T_{|\lambda|}(\overbrace{U, U, \ldots, U}^{m_{|\lambda|}(\mu)})}{m_{|\lambda|}(\mu)!}
\end{eqnarray*}
\end{proposition}
\begin{proof}
We decompose $\mathcal{S}^\lambda$ into a linear combination of virtual representations, each having character equal to an indicator function of some cycle type $\mu$. The scalar multiples are $\chi_\mu^\lambda$. By Lemma \ref{class_ind}, $\frac{T_i(U, U, \ldots, U)}{m_i(\mu)!}$ corresponds to the indicator function of cycle type $(i^{m_i(\mu)})$. Multiplying these together corresponds to taking the tensor product of class functions on $S_{im_i(\mu)}$ and inducing up to $S_n$, which precisely gives the indicator function of cycle type $\mu$.
\end{proof}

\begin{remark} \label{adjointRmk}
Suppose that we wish to decompose the expression in Proposition \ref{semiirred} into $X_\mu$. It is clear that if $U \neq \mathbf{1}$ then we get the definition of $X_{\vv{\mu}}$ where $\vv{\mu}(U) = \lambda$ and $\vv{\mu}(V)$ is the trivial partition for $V \neq U$. If $U = \mathbf{1}$, we use Proposition \ref{Wprop2} to see that we must describe $\Ind{S_{|\lambda|} \times S_m}{S_{|\lambda|+m}} (\mathcal{S}^{\lambda} \boxtimes \mathbf{1}_{S_m})$ for $m$ sufficiently large. Under (the inverse of) the characteristic map between symmetric functions and representations of symmetric groups, calculating the induced representation amounts to calculating the product of symmetric functions $s_\lambda h_m$, which is described combinatorially by the Pieri rule. The Pieri rule states that we get $\sum_\mu s_\mu$ where the sum across all partitions $\mu$ obtained from $\lambda$ by adding $m$ boxes, no two in the same column. For example, suppose $\lambda = (2,1)$ and $m = 4$. The valid $\mu$ are shown below, where the added boxes are highlighted.

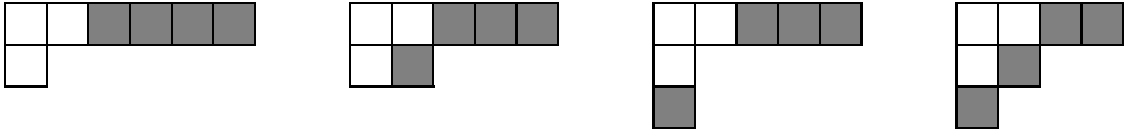
\begin{figure}[H]
\centering
\ytableausetup{nosmalltableaux}
\begin{ytableau}
*(white) &*(white) &*(gray) &*(gray) & *(gray) &*(gray) \\
*(white)
\end{ytableau}
\hspace{10mm}
\begin{ytableau}
*(white) &*(white) &*(gray) &*(gray) & *(gray) \\
*(white) & *(gray)
\end{ytableau}
\hspace{10mm}
\begin{ytableau}
*(white) &*(white) &*(gray) &*(gray) & *(gray) \\
*(white) \\
*(gray)
\end{ytableau}
\hspace{10mm}
\begin{ytableau}
*(white) &*(white) &*(gray) &*(gray)  \\
*(white) & *(gray) \\
*(gray)
\end{ytableau}
\end{figure}

\noindent
When $m$ is larger than the number of columns in the diagram of $\lambda$ (i.e. the longest part of $\lambda$), there is no restriction on the collection of columns that a box may be added to, save that the final result must be a partition.
We are interested in the set of partitions obtained by removing the first row of each of the diagrams after performing the above operation. The operation of removing the top row is the same as removing a box from each column. Thus, we are interested in all partitions obtained by adding at most one box to some column, and then removing one box from each column. This is equivalent to removing one box from each of the columns in the diagram of $\lambda$ that were not chosen. In other words, the set we are interested in consists of all partitions obtained from $\lambda$ by removing some number of boxes, no two in the same column. If we write $h_r^{\perp}$ for the operator adjoint to multiplication by $h_r$ with respect to the usual bilinear form form on $\Lambda$, then by the Pieri rule, $h_r^{\perp}s_\lambda$ is $\sum_\mu s_\mu$ across all partitions $\mu$ obtained from $\lambda$ by removing $r$ boxes in the diagram of $\lambda$, no two in the same column. Continuing to encode partitions as their associated Schur functions, we find that the desired decomposition is
\[
\left( \sum_{r=0}^{\infty} h_r^{\perp} \right) s_{\lambda}
\]
This is because, for $m$ sufficiently large, there is no restriction on the number of boxes that could be removed (i.e. $r=0,1,2,\cdots$).
\end{remark}
\begin{example} \label{trivexamp}
Suppose that in Proposition \ref{semiirred}, $\lambda = (1^r)$ and $U = \mathbf{1}$. Since $h_i^{\perp}s_{(1^r)} = 0$ for $i \geq 2$, and $h_1^{\perp}s_{(1^r)} = s_{(1^{r-1})}$ (and $h_0^\perp s_{(1^r)} = s_{(1^r)}$), we obtain
\[
\lim_{m \to \infty} \Ind{S_{r} \times S_m}{S_{r+m}} \left(\left( \mathbf{1}^{\boxtimes r}\otimes \mathcal{S}^{(1^r)}\right) \boxtimes \left(\mathbf{1}^{\boxtimes m} \otimes \mathbf{1}_{S_m}\right)\right) = X_{\vv{\mu}^{(1)}} + X_{\vv{\mu}^{(2)}}
\]
where $\vv{\mu}^{(1)}(\mathbf{1}) = (1^r)$ and $\vv{\mu}^{(2)}(\mathbf{1}) = (1^{r-1})$ (and all parts of these multipartitions associated to $U \neq \mathbf{1}$ are the empty partition).
\end{example}

\begin{remark}\label{adjointRmk2}
To reconstruct $X_{\vv{\lambda}}$ from the objects in Proposition \ref{semiirred}, we need to invert the operator $\sum_{r=0}^{\infty} h_r^{\perp}$. Recognising it as the adjoint of $H(1)$ (where $H(t) = h_0 + h_1 t+ h_2 t^2 + \cdots$ is the generating function of complete symmetric functions), we may write the inverse as the adjoint of $E(-1)$ (recall that $H(t)E(-t) = 1$). So, the relevant operator (when we are encoding partitions as Schur functions) is $\sum_{r=0}^{\infty} (-1)^re_{r}^{\perp}$ (where $e_r^\perp$ is adjoint to multiplication by $e_r$).
\end{remark}

\begin{proposition} \label{PtoT}
Let $U$ be an object of $\mathcal{C}$, and $\varphi_U: \Lambda \otimes_{\mathbb{Z}}\mathbb{Q} \to \mathcal{G}_\infty(\mathcal{C})$ be defined by 
\[
\varphi_U(e_i) = \lim_{m \to \infty} \Ind{S_{i} \times S_m}{S_{i+m}} \left(\left(U^{\boxtimes i} \otimes \mathcal{S}^{(1^i)}\right) \boxtimes \left(\mathbf{1}^{\boxtimes m} \otimes \mathbf{1}_{S_m}\right)\right)
\]
so that if $U \neq \mathbf{1}$ is a simple object, then $\varphi_U(e_i)$ is the basic hook $\overline{e}_i(U)$, whilst Example \ref{trivexamp} shows that if $U = \mathbf{1}$, then $\varphi(e_i) = \overline{e}_i(\mathbf{1}) + \overline{e}_{i-1}(\mathbf{1})$. 
Then:
\[
\varphi_U(p_n) = \sum_{d | n} d T_d(U^{\frac{n}{d}})
\]
\end{proposition}
\begin{proof}
We use the generating functions $E(t) = \sum_{n=0}^{\infty} e_n t^n$, and $P(t) = \sum_{n=0}^{\infty} p_{n+1}t^n$. Recall that
\[
\frac{d}{dt} \log(E(t)) = P(-t)
\]
Additionally, we have the following expression using Proposition \ref{semiirred}, following directly from the character formula for the sign representation ($\chi_{\mu}^{(1^{|\mu|})} = \varepsilon_\mu$).
\begin{eqnarray*}
\varphi_U(e_n) &=& \sum_{\lambda \vdash n} {\varepsilon_\lambda} \frac{T_1(\overbrace{U, U, \ldots, U}^{m_1(\lambda)})}{m_1(\lambda)!} \frac{T_2(\overbrace{U, U, \ldots, U}^{m_2(\lambda)})}{m_2(\lambda)!} \cdots \frac{T_{n}(\overbrace{U, U, \ldots, U}^{m_{n}(\lambda)})}{m_{n}(\lambda)!} \\
&=&
\sum_{\lambda \vdash n} \prod_{i=1}^{n} \frac{(-1)^{m_i(\lambda)(i-1)}}{m_i(\lambda)!} T_i(\overbrace{U,U,\ldots,U}^{m_i(\lambda)}) \\
&=&
\sum_{\lambda \vdash n} \prod_{i=1}^{n} {(-1)^{m_i(\lambda)(i-1)}} \sum_{\mu^{(i)} \vdash m_i(\lambda)} \frac{\varepsilon_{\mu^{(i)}}}{z_{\mu^{(i)}}}\prod_j T_i(U^j)^{m_j(\mu^{(i)})}
\end{eqnarray*}
In the last step, we used Corollary \ref{usefulCor}. We now calculate the generating function $E(t)$. Below we abbreviate $m_i(\lambda)$ to $m_i$, and use the fact that $\lambda$ is parametrised by the numbers $m_i$.
\begin{eqnarray*}
\varphi_U(E(t)) &=& \sum_{n=0}^{\infty}
\sum_{\lambda \vdash n} t^{|\lambda|}\prod_{i=1}^{n} {(-1)^{m_i(\lambda)(i-1)}} \sum_{\mu^{(i)} \vdash m_i(\lambda)} \frac{\varepsilon_{\mu^{(i)}}}{z_{\mu^{(i)}}} \prod_j T_i(U^j)^{m_j(\mu^{(i)})}
 \\
&=& \sum_{m_1 = 0}^{\infty} \sum_{m_2 = 0}^{\infty} \cdots 
\left( 
\prod_{i=1}^{\infty} t^{i m_i }(-1)^{(i-1)m_i} \sum_{\mu^{(i)} \vdash m_i} \frac{\varepsilon_{\mu^{(i)}}}{z_{\mu^{(i)}}} \prod_j T_i(U^j)^{m_j(\mu^{(i)})}
\right) \\
&=&
\prod_{i =1}^{\infty}
 \sum_{m_i = 0}^{\infty}
\left( 
t^{i m_i }(-1)^{(i-1)m_i} \sum_{\mu^{(i)} \vdash m_i} \frac{\varepsilon_{\mu^{(i)}}}{z_{\mu^{(i)}}} \prod_j T_i(U^j)^{m_j(\mu^{(i)})}
\right) \\
&=&
\prod_{i =1}^{\infty}
 \sum_{m_i = 0}^{\infty}
\left( 
t^{i m_i }(-1)^{(i-1)m_i} \sum_{\mu^{(i)} \vdash m_i}  \prod_j \frac{T_i(U^j)^{m_j(\mu^{(i)})}(-1)^{m_j(\mu^{(i)})(j-1)}}{m_j(\mu^{(i)})!j^{m_j(\mu^{(i)})}}
\right) \\
&=&
\prod_{i =1}^{\infty}
\left( 
 \sum_{m_i = 0}^{\infty}
 \sum_{\mu^{(i)} \vdash m_i}  
 \prod_{j=1}^{\infty}
t^{i jm_j(\mu^{(i)}) }(-1)^{(i-1)jm_j(\mu^{(i)})} \frac{T_i(U^j)^{m_j(\mu^{(i)})}(-1)^{m_j(\mu^{(i)})(j-1)}}{m_j(\mu^{(i)})!j^{m_j(\mu^{(i)})}}
\right) \\
&=&
\prod_{i =1}^{\infty}
\left( 
 \sum_{m_i = 0}^{\infty}
 \sum_{\mu^{(i)} \vdash m_i} 
 \prod_{j=1}^{\infty}
\frac{1}{m_j(\mu^{(i)})!} \left(t^{i j }(-1)^{(ij-1)} \frac{T_i(U^j)}{j}\right)^{m_j(\mu^{(i)})}
\right) \\
&=&
\prod_{i =1}^{\infty}
\left( 
 \sum_{\mu^{(i)} \in \mathcal{P}} 
 \prod_{j=1}^{\infty}
\frac{1}{m_j(\mu^{(i)})!} \left(t^{i j }(-1)^{(ij-1)} \frac{T_i(U^j)}{j}\right)^{m_j(\mu^{(i)})}
\right) \\
&=&
\prod_{i =1}^{\infty}
\left( 
\exp\left(
\sum_{j=1}^{\infty}
t^{i j }(-1)^{(ij-1)} \frac{T_i(U^j)}{j}\right)
\right) 
\end{eqnarray*}
In the second last step, we used the fact that the set of integer partitions $\mu$ is parametrised by the numbers $m_j(\mu) \in \mathbb{Z}_{\geq 0}$ with all but finitely many being zero. This allows us to sum over the numbers $m_j(\mu^{(i)})$ independently of each other.
Now we may take the derivative of the logarithm with respect to $t$:
\begin{eqnarray*}
\varphi_U\left(\frac{E'(t)}{E(t)}\right) &=& \sum_{i=1}^{\infty}\sum_{j=1}^{\infty} ij t^{i j - 1 }(-1)^{(ij-1)} \frac{T_i(U^j)}{j} \\
&=& \sum_{i=1}^{\infty}\sum_{j=1}^{\infty} i (-t)^{ij-1} T_i(U^j) \\
&=& \sum_{n=1}^{\infty} \sum_{d | n} d (-t)^{n-1} T_d(U^{n/d})
\end{eqnarray*}
In the last step, the change of variables $i = d$, $ij = n$ was used. Equating the terms of the power series with those of $\varphi_U(P(-t))$ gives the result.
\end{proof}
\subsection{Generating Function for Irreducibles}
\noindent
We now prove the main theorem of this paper. It provides an generating function for the basis $X_{\vv{\lambda}}$ in terms of the $T_n(U)$ generators. In principle this gives a way to decompose products $X_{\vv{\mu}}X_{\vv{\nu}}$, and therefore a way to calculate multiplicities of tensor products in wreath product Deligne categories (this calculation is carried out in Section \ref{calc_section}, and Theorem \ref{coefthm} in particular).
\begin{theorem} \label{bigthm}
Write $\Lambda_{\mathbb{Q}}^{(U)}$ for a copy of the ring of symmetric functions with rational coefficients, whose variables we associate with $U \in I(\mathcal{C})$. If $f$ is a symmetric function, we write $f^{(U)}$ to denote $f$ considered as an element of $\Lambda_{\mathbb{Q}}^{(U)}$. We work in $\left(\bigotimes_{U \in I(\mathcal{C})} \Lambda_{\mathbb{Q}}^U\right) \hat{\otimes} \mathcal{G}(\mathcal{C})$, the completed tensor product of $\left(\bigotimes_{U \in I(\mathcal{C})} \Lambda_{\mathbb{Q}}^U\right)$ with $\mathcal{G}(\mathcal{C})$. If $c_\mu^{(U)}$ ($\mu \in \mathcal{P}$, $U \in I(\mathcal{C})$) are constants, let
\[
T_l\left(\sum_{\mu \in \mathcal{P}, U \in I(\mathcal{C})} c_\mu^{(U)} p_\mu^{(U)} [U]\right) = \sum_{\mu \in \mathcal{P}, U \in I(\mathcal{C})} c_\mu^{(U)} p_\mu^{(U)} \otimes T_l(U)
\]
which is an element of $\left(\bigotimes_{U \in I(\mathcal{C})} \Lambda_{\mathbb{Q}}^U\right) \hat{\otimes} \mathcal{G}_\infty(\mathcal{C})$. We have the following equality:
\[
\sum_{\vv{\lambda} \in \mathcal{P}^\mathcal{C}} \left( \prod_{U \in I(\mathcal{C})} s_{\vv{\lambda}(U)}^{(U)} \right) \otimes X_{\vv{\lambda}} =
\left(\sum_{r \geq 0} (-1)^r e_r^{(\mathbf{1})} \right)
\prod_{l=1}^{\infty}\exp \left(T_l \left(\log \left(1 + \sum_{U \in I(\mathcal{C})} p_l^{(U)}[U]\right) \right) \right)
\]
\end{theorem}
\begin{proof}
We firstly note that the first factor on the right hand side can be inverted (using the generating function relation $H(t)E(-t) = 1$). Moving it to the left had side it acts as an operator on the symmetric functions, but by taking the adjoint, we may make it act of the $X_{\vv{\lambda}}$. Taking into consideration Remark \ref{adjointRmk2}, we see that the effect of this manipulation is to replace $X_{\vv{\lambda}}$ with the expression in Proposition \ref{semiirred}: $\lim_{m \to \infty} \Ind{S_{\vv{\lambda}} \times S_m}{S_{|\vv{\lambda}|+m}} \left(\boxtimes_{U \in I(\mathcal{C})} \left(U^{\boxtimes |\vv{\lambda}|} \otimes \mathcal{S}^{\vv{\lambda}(U)}\right) \boxtimes \left(\mathbf{1}^{\boxtimes m} \otimes \mathbf{1}_{S_m}\right)\right)$. The left hand side of the equation becomes
\[
\sum_{\vv{\lambda} \in \mathcal{P}^\mathcal{C}} \left( \prod_{U \in I(\mathcal{C})} s_{\vv{\lambda}(U)}^{(U)} \right) \otimes \left( \lim_{m \to \infty} \Ind{S_{\vv{\lambda}} \times S_m}{S_{|\vv{\lambda}|+m}} \left(\boxtimes_{U \in I(\mathcal{C})} \left(U^{\boxtimes |\vv{\lambda}(U)|} \otimes \mathcal{S}^{\vv{\lambda}(U)}\right) \boxtimes \left(\mathbf{1}^{\boxtimes m} \otimes \mathbf{1}_{S_m}\right)\right)\right) 
\]
and we are required to prove that it is equal to
\[
\prod_{l=1}^{\infty}\exp \left(T_l\left(\log \left(1 + \sum_{U} p_l^{(U)}[U]\right) \right) \right)
\]
Using the same method as in the proof of Proposition \ref{semiirred}, we seek to write the expression in terms of the elements $T_i(U)$. We have
\begin{eqnarray*}
& &\lim_{m \to \infty} \Ind{S_{\vv{\lambda}} \times S_m}{S_{|\vv{\lambda}|+m}} \left(\boxtimes_{U \in I(\mathcal{C})}  \left(U^{\boxtimes |\vv{\lambda}(U)|} \otimes \mathcal{S}^{\vv{\lambda}(U)}\right) \boxtimes \left(\mathbf{1}^{\boxtimes m} \otimes \mathbf{1}_{S_m}\right) \right) \\
&=& 
\sum_{\mu^{(1)} \in \mathcal{P}} \sum_{\mu^{(2)} \in \mathcal{P}} \cdots
\left( \prod_{U_i \in I(\mathcal{C})} \chi_{\mu^{(i)}}^{\vv{\lambda}(U_i)} \right)
\prod_{i=1}^{\infty}\frac{T_i(\overbrace{U_1, U_1, \ldots}^{m_i(\mu^{(1)})}, \overbrace{U_2, U_2, \ldots}^{m_i(\mu^{(2)})}, \ldots)}{m_1(\mu^{(1)})! m_1(\mu^{(2)})! \cdots}
\end{eqnarray*}
Here $\mu^{(i)} \in \mathcal{P}$ describes a cycle type in a symmetric group associated to $U_i \in I(\mathcal{C})$.
We now let $\nu  = \cup_i \mu^{(i)}$ and use Proposition \ref{tdecomp} to express our equation in terms of $T_i(U)$ (i.e. without any inductions). We obtain
\begin{eqnarray*}
\sum_{\mu^{(1)}} \sum_{\mu^{(2)}} &\cdots&
\left( \prod_{U_i \in I(\mathcal{C})} \chi_{\mu^{(i)}}^{\vv{\lambda}(U_i)} \right)
\left( \prod_{j=1}^{\infty} \frac{1}{m_j(\mu^{(1)})! m_j(\mu^{(2)})! \cdots} \right)
 \\
\times & &
\sum_{\alpha^{(1)} \vdash m_1(\nu)}
\sum_{\alpha^{(2)} \vdash m_2(\nu)}
\cdots 
\left(\frac{\varepsilon_{\alpha^{(1)}}}{z_{\alpha^{(1)}}}
\frac{\varepsilon_{\alpha^{(2)}}}{z_{\alpha^{(2)}}}
\cdots \right) 
\prod_{l=1}^{\infty}
\sum_{\sigma \in S_{m_l(\nu)}} T_{l,\alpha^{(l)}}(\sigma(\overbrace{U_{1}, U_1, \ldots}^{m_l(\mu^{(1)})}, \overbrace{U_2, U_2, \ldots}^{m_l(\mu^{(2)})}, \ldots ))
\end{eqnarray*}
Ultimately we are calculating a generating function whose inner product with the symmetric function $\prod_{U \in I(\mathcal{C})}s_{\vv{\lambda}(U)}^{(U)}$ is the above quantity. If $\sigma$ and $\rho$ are partitions, then the fact that the inner product of $s_\sigma$ and $p_\rho$ is $\chi_{\rho}^{\sigma}$ allows us to replace the character values with power-sum symmetric functions and sum over all possible power-sum symmetric functions. We also note that the subgroup $S_{m_i(\mu^{(1)})} \times S_{m_i(\mu^{(2)})} \times \cdots$ of $S_{m_l(\nu)}$ acts fixes the vector $(\overbrace{U_{1}, U_1, \ldots}^{m_l(\mu^{(1)})}, \overbrace{U_2, U_2, \ldots}^{m_l(\mu^{(2)})}, \ldots )$ (with the usual permutation action). This means that we may restrict the sum to coset representatives of this subgroup at the cost of multiplying by $m_i(\mu^{(1)})! m_i(\mu^{(2)})! \cdots$ (which cancels out the denominators following the symmetric group characters in the above expression). Our new expression is:
\begin{eqnarray*}
\sum_{\mu^{(1)}} \sum_{\mu^{(2)}} &\cdots&
\left( \prod_{U_i \in I(\mathcal{C})} p_{\mu^{(i)}}^{(U_i)} \right) \otimes
\sum_{\alpha^{(1)} \vdash m_1(\nu)}
\sum_{\alpha^{(2)} \vdash m_2(\nu)}
\cdots 
\left(\frac{\varepsilon_{\alpha^{(1)}}}{z_{\alpha^{(1)}}}
\frac{\varepsilon_{\alpha^{(2)}}}{z_{\alpha^{(2)}}}
\cdots \right)  \\
\times & &
\prod_{l=1}^{\infty}
\sum_{\sigma \in S_{m_l(\nu)}/S_{m_l(\mu^{(1)})} \times S_{m_l(\mu^{(2)})} \cdots} T_{l,\alpha^{(l)}}(\sigma(\overbrace{U_{1}, U_1, \cdots}^{m_l(\mu^{(1)})}, \overbrace{U_2, U_2, \cdots}^{m_l(\mu^{(2)})}, \cdots ))
\end{eqnarray*}
\noindent
Now, note that the $T_{l,\alpha^{(l)}}(\cdots)$ are summed over all distinct reorderings of their arguments. We now inspect the sum over $\alpha^{(l)}$ and $\sigma$ more closely (which we take to include the terms $(p_l^{(U_1)})^{m_l(\mu^{(1)})}(p_l^{(U_2)})^{m_l(\mu^{(2)})}\cdots$ coming from $\prod_{U_i \in I(\mathcal{C})} p_{\mu^{(i)}}^{(U_i)}$):
\begin{eqnarray*}
& &(p_l^{(U_1)})^{m_l(\mu^{(1)})}(p_l^{(U_2)})^{m_l(\mu^{(2)})}\cdots \otimes
\sum_{\alpha^{(l)} \vdash m_l(\nu)} \frac{\varepsilon_{\alpha^{(l)}}}{z_{\alpha^{(l)}}}
\\ \times & &
\sum_{\sigma \in S_{m_l(\nu)}/S_{m_l(\mu^{(1)})} \times S_{m_l(\mu^{(2)})} \cdots} T_{l, \alpha^{(l)}}(\sigma(\overbrace{U_{1}, U_1, \ldots}^{m_l(\mu^{(1)})}, \overbrace{U_2, U_2, \ldots}^{m_l(\mu^{(2)})}, \ldots ))
\end{eqnarray*}
Recalling that $\frac{\varepsilon_{\alpha^{(l)}}}{z_{\alpha^{(l)}}} = \prod_{j=1}^{\infty} \frac{(-1)^{m_j(\alpha^{(l)})(j-1)}}{m_j(\alpha^{(l)})! j^{m_j(\alpha^{(l)})}}$, we have
\begin{eqnarray*}
& &\left( \prod_{U_i \in I(\mathcal{C})} (p_l^{(U_i)})^{m_l(\mu^{(i)})}
\right) \otimes
\sum_{\alpha^{(l)} \vdash m_l(\nu)}  \prod_{j=1}^{\infty} \frac{(-1)^{m_j(\alpha^{(l)})(j-1)}}{m_j(\alpha^{(l)})! j^{m_j(\alpha^{(l)})}}
\\ \times & &
\sum_{\sigma \in S_{m_l(\nu)}/S_{m_l(\mu^{(1)})} \times S_{m_l(\mu^{(2)})} \cdots} T_{l,\alpha^{(l)}}(\sigma(\overbrace{U_{1}, U_1, \ldots}^{m_l(\mu^{(1)})}, \overbrace{U_2, U_2, \ldots}^{m_l(\mu^{(2)})}, \ldots ))
\end{eqnarray*}
\noindent
We sum over all possible values of $m_l(\mu^{(1)}), m_l(\mu^{(2)}), \ldots$, which means that $\sigma(\overbrace{U_{1}, U_1, \ldots}^{m_l(\mu^{(1)})}, \overbrace{U_2, U_2, \ldots}^{m_l(\mu^{(2)})}, \ldots )$ varies across all finite words $W$ in the $U_i$ without repetition. To calculate $T_{l, \alpha^{(l)}}(W)$ we write $W_{\alpha^{(l)}, r}$ for the product of the letters of the subword of $W$ starting at the $(\alpha_1^{(l)}+\alpha_2^{(l)}+\cdots + \alpha_{r-1}^{(l)}+1)$-th place and finishing at the $(\alpha_1^{(l)} + \alpha_2^{(l)} + \cdots + \alpha_r^{(l)})$-th place. This lets us write (by definition of $T_{m, \lambda}$)
\[
T_{l, \alpha^{(l)}}(W) = T_{l}(W_{\alpha^{(l)}, 1}) T_{l}(W_{\alpha^{(l)}, 2}) \cdots T_{l}(W_{\alpha^{(l)}, l(\alpha^{(l)})})
\]
Now, in $\mathcal{G}(\mathcal{C})$ we may write $[W_{\alpha^{(l)}, r}] = \sum_{U \in I(\mathcal{C})} M_{W, \alpha^{(l)}, r}^{U}[U]$, and because $T_l(-)$ is linear, 
\[
T_{l}(W_{\alpha^{(l)}, r}) = \sum_{U \in I(\mathcal{C})} M_{W, \alpha^{(l)}, r}^{(U)} T_l(U) 
\]
If we write $|W|$ for the length of the word $W$, and $n_U(W)$ for the number of occurrences of $U$ in $W$, then we may rewrite our earlier expression as
\begin{eqnarray*}
& &\sum_{W} \left( \prod_{U_i \in I(\mathcal{C})} (p_l^{(U_i)})^{n_U(W)}
\right) \otimes
\left(
\sum_{\alpha^{(l)} \vdash |W|}  \prod_{j=1}^{\infty} \frac{(-1)^{m_j(\alpha^{(l)})(j-1)}}{m_j(\alpha^{(l)})! j^{m_j(\alpha^{(l)})}} \right) \left( \prod_r T_{l}(W_{\alpha^{(l)}, r}) \right)
\\  &=&
\sum_{W} 
\left(
\sum_{\alpha^{(l)} \vdash |W|}  \prod_{j=1}^{\infty} \frac{(-1)^{m_j(\alpha^{(l)})(j-1)}}{m_j(\alpha^{(l)})! j^{m_j(\alpha^{(l)})}} \times  \prod_r  \left( \prod_{U \in I(\mathcal{C})} (p_l^{(U)})^{n_U(W_{\alpha^{(l)},r})}
\right) \otimes
\prod_r T_{l}(W_{\alpha^{(l)}, r})
 \right)
\end{eqnarray*}
\noindent
We now note that each $W_{\alpha^{(l)},r}$ varies independently over all words in the $U_i$ of length $\alpha_r^{(l)}$. We may therefore remove the sum over $W$ at the cost of replacing 
\[
\prod_r \left( \prod_{U \in I(\mathcal{C})} (p_l^{(U)})^{n_U(W_{\alpha^{(l)},r})}\right) \otimes \prod_r T_{l}(W_{\alpha^{(l)}, r})
\]
with
\[
T_l \left(\left(\sum_{U \in I(\mathcal{C})} p_l^{(U)} [U]\right)^{\alpha_r^{(l)}}\right) 
\]
This leaves us with
\begin{eqnarray*}
& &
\sum_{\alpha^{(l)} \in \mathcal{P}} \left( \prod_{j=1}^{\infty} \frac{(-1)^{m_j(\alpha^{(l)})(j-1)}}{m_j(\alpha^{(l)})! j^{m_j(\alpha^{(l)})}} \right) 
\left( \prod_r  
T_l \left(\sum_{U \in I(\mathcal{C})} p_l^{(U)} [U]\right)^{\alpha_r^{(l)}} \right)\\
 &=&
\sum_{\alpha^{(l)} \in \mathcal{P}}  \prod_{j=1}^{\infty} \frac{1}{m_j(\alpha^{(l)})!}\left(
\frac{(-1)^{(j-1)}}{j}
\right)^{m_j(\alpha^{(l)})}
\left( T_l \left(\sum_{U \in I(\mathcal{C})} p_l^{(U)} [U]\right)^{j}
 \right)^{m_j(\alpha^{(l)})} \\
 &=&
  \prod_{j=1}^{\infty} \sum_{m_j(\alpha^{(l)})=0}^{\infty} \frac{1}{m_j(\alpha^{(l)})!}\left(
\frac{(-1)^{(j-1)}}{j}
\right)^{m_j(\alpha^{(l)})}
\left( T_l \left(\sum_{U \in I(\mathcal{C})} p_l^{(U)} [U]\right)^{j}
 \right)^{m_j(\alpha^{(l)})} 
\end{eqnarray*}
Here we used the fact that summing over all partitions $\alpha^{(l)}$ is equivalent to summing over all possible values of $m_r(\alpha^{(l)})$ for all $r$. Now we recognise the power series for the exponential and then for the logarithm.
\begin{eqnarray*}
& & \prod_{j=1}^{\infty} \exp\left(
\frac{(-1)^{(j-1)}}{j} T_l \left(\sum_{U \in I(\mathcal{C})} p_l^{(U)} [U]\right)^{j} \right) \\
 &=& \exp\left(
 \sum_{j=1}^{\infty}\frac{(-1)^{(j-1)}}{j} T_l \left(\sum_{U \in I(\mathcal{C})} p_l^{(U)} [U]\right)^{j} \right) \\
 &=&\exp\left(
T_l \left( \log\left(1+\sum_{U \in I(\mathcal{C})} p_l^{(U)} [U]\right) \right) \right)
\end{eqnarray*}
Now we simply multiply this expression for $l \in \mathbb{Z}_{>0}$ to obtain the desired result (since $T_{l_1}(U)$ commutes with $T_{l_2}(V)$ whenever $l_1 \neq l_2$, we do not need to be careful about commuting exponentials).
\end{proof}
\noindent
In order to obtain expressions for $X_{\vv{\lambda}}$ in terms of basic hooks, we must write $T_n(U)$ in terms of basic hooks.

\begin{proposition} \label{TtoP}
Recall the setting of Proposition \ref{PtoT}, where for any object $U$ of $\mathcal{C}$, we had
\[
\varphi_U(p_n) = \sum_{d | n} d T_d(U^{\frac{n}{d}})
\]
Let $\mu(n)$ be the M\"{o}bius function (defined on positive integers by $\sum_{d | n} \mu(d) = \delta_{n,1}$). We have:
\[
T_r(U) = \frac{1}{r} \sum_{d | r} \varphi_{(U^{\frac{r}{d}})}(p_d)\mu(r/d)
\]
\end{proposition}
\begin{proof}
We directly calculate
\begin{eqnarray*}
\frac{1}{r} \sum_{d | r} \varphi_{(U^{\frac{r}{d}})}(p_d)\mu(r/d) &=& \frac{1}{r} \sum_{d|r} \mu(r/d) \left( \sum_{d' | d} d'T_{d'}((U^{\frac{r}{d}})^{\frac{d}{d'}}) \right) \\
&=& \frac{1}{r} \sum_{d | r} \mu(r/d) \left(\sum_{d'|d} d' T_{d'}(U^{\frac{r}{d'}}) \right) \\
&=& \frac{1}{r} \sum_{d' | r} \left( \sum_{d' | d | r} \mu(r/d)\right) d' T_{d'}(U^{\frac{r}{d'}}) \\
&=& \frac{1}{r} \sum_{d' | r} \delta_{d',r} d' T_{d'}(U^{\frac{r}{d'}}) \\
&=& T_{r}(U) 
\end{eqnarray*}
\end{proof}
\noindent
This means that to express the $T_r(U)$ in terms of basic hooks, it is enough to decompose $\varphi_{(U^{\frac{r}{d}})}(p_d)$ into basic hooks. This task is complicated by the fact that $\varphi_{(V)}(p_d)$ is not linear in $V$ for $d > 1$. However, this difficulty is mitigated if $U^{\frac{r}{d}}$ is itself a simple object of $\mathcal{C}$, for example when $\mathcal{C} = kG-mod$ where $G$ is an abelian group (simple objects are precisely one dimensional representations of $G$, the set of which is closed under taking tensor products). When $\mathcal{C} = kG-mod$ (for abelian $G$) the problem amounts to expressing power sum symmetric functions in terms of elementary symmetric functions. The elementary symmetric functions give rise to basic hooks for simple $U \neq \mathbf{1}$, and to a sum of two basic hooks for $U = \mathbf{1}$, as per Proposition \ref{PtoT}.

\section{Applications to Symmetric Groups and Wreath Products} \label{calc_section}
\noindent
We discuss a selection of results about the asymptotic representation theory of symmetric groups and wreath products that follow from our results. Recall that the Deligne category $\underline{\mbox{Rep}}(S_t)$ is a tensor category that can be thought of as an ``interpolation'' of the representation categories of finite symmetric groups. 
\begin{theorem}
The ring $\mathcal{G}_\infty(\mathcal{C})$ is isomorphic to the Grothendieck ring (with rational coefficients) of the wreath product version of the Deligne category, $S_t(\mathcal{C})$, when $t \notin \mathbb{Z}_{\geq 0}$. The Grothendieck ring with integral coefficients is isomorphic to the integral version of $\mathcal{G}_\infty(\mathcal{C})$ described in Remark \ref{integral_rmk}.
\end{theorem}
\begin{proof}
When $t \notin \mathbb{Z}_{\geq 0}$, the simple objects of the category $S_t(\mathcal{C})$ are parametrised by $\vv{\lambda} \in \mathcal{P}^\mathcal{C}$. The methods of Theorem \ref{delconst} allow one to deduce that the structure constants for non-integral $t$ agree with the corresponding stable limits as $t \in \mathbb{Z}_{\geq 0}$ tends to infinity.
\end{proof}
\noindent
The wreath product categories are discussed in \cite{Mori}, and various aspects of the theory of Deligne categories are discussed in \cite{etingofRCR1} and \cite{etingofRCR2}.
\noindent
We now give a way for computing a formula for structure constants of $\mathcal{G}_\infty(\mathcal{C})$ with respect to the $X_{\vv{\lambda}}$ basis. Of course, these are also the structure constants in the Grothendieck ring of a Deligne category. We use Theorem \ref{bigthm} with multiple different sets of symmetric function variables. It will be convenient to write $p_l(\mathbf{x}^{(U)})$ instead of $p_l^{(U)}(\mathbf{x})$. 
\begin{theorem} \label{coefthm}
Write $N_{U, V}^{W}$ for the structure tensor (so that $[U][V] = \sum_{W} N_{U,V}^{W}$). Write $\mathbf{z}^{(U)}$ to denote the family of symmetric function variables $\bigoplus_{V_1, V_2 \in I(\mathcal{C})} \left(\mathbf{x}^{(V_1)}\mathbf{y}^{(V_2)}\right)^{\oplus N_{V_1, V_2}^{U}}$, where direct sum notation denotes a disjoint union of symmetric function variables, and the direct sum in the exponents denotes the multiplicity of each of the sets of variables. Then the the multiplicity of $X_{\vv{\lambda}}$ in $X_{\vv{\mu}}X_{\vv{\nu}}$ is given by the coefficient of
\[
\left( \prod_{U \in I(\mathcal{C})} s_{\vv{\mu}(U)}(\mathbf{x}^{(U)}) \prod_{V \in I(\mathcal{C})} s_{\vv{\nu}(V)}(\mathbf{y}^{(V)}) \right) 
\]
in
\[
\prod_{U,V \in I(\mathcal{C})} \left( \sum_{\rho \in \mathcal{P}} s_\rho(\mathbf{x}^{(U)}) s_\rho(\mathbf{y}^{(V)})
\right)^{N_{U,V}^{(\mathbf{1})}}
 \left( \prod_{U \in I(\mathcal{C})} s_{\vv{\lambda}(U)}\left(\mathbf{x}^{(U)},\mathbf{y}^{(U)},\mathbf{z}^{(U)} \right) \right)
\]
\end{theorem}
\begin{proof}
We manipulate generating functions, starting with one where the coefficient of 
\[
\left( \prod_{U \in I(\mathcal{C})} s_{\vv{\mu}(U)}(\mathbf{x}^{(U)}) \prod_{V \in I(\mathcal{C})} s_{\vv{\nu}(V)}(\mathbf{y}^{(V)}) \right)
\]
 is $X_{\vv{\mu}}X_{\vv{\nu}}$. Thus, the problem reduces to understanding the coefficient of $X_{\vv{\lambda}}$ in the generating function.
\begin{eqnarray*}
& &
\sum_{\vv{\mu} \in \mathcal{P}^\mathcal{C}}\sum_{\vv{\nu} \in \mathcal{P}^\mathcal{C}}
\left( \prod_{U \in I(\mathcal{C})} s_{\vv{\mu}(U)}(\mathbf{x}^{(U)}) \prod_{V \in I(\mathcal{C})} s_{\vv{\nu}(V)}(\mathbf{y}^{(V)}) \right) \otimes \left( X_{\vv{\mu}}X_{\vv{\nu}} \right) \\
&=&
\left( \sum_{\vv{\mu} \in \mathcal{P}_\mathcal{C}} \left( \prod_{U \in I(\mathcal{C})} s_{\vv{\mu}(U)}(\mathbf{x}^{(U)}) \right) \otimes X_{\vv{\mu}}  \right)
\left( \sum_{\vv{\nu} \in \mathcal{P}_\mathcal{C}} \left( \prod_{V \in I(\mathcal{C})} s_{\vv{\nu}(V)}(\mathbf{y}^{(V)}) \right) \otimes X_{\vv{\nu}}  \right) \\
&=&
\left(\sum_{r \geq 0} (-1)^r e_r(\mathbf{x}^{(\mathbf{1})}) \right)
\prod_{l=1}^{\infty}\exp \left(T_l \left( \log \left(1 + \sum_{U \in I(\mathcal{C})} p_l(\mathbf{x}^{(U)})[U]\right)\right) \right) \\
&\times&
\left(\sum_{r \geq 0} (-1)^r e_r(\mathbf{y}^{(\mathbf{1})}) \right)
\prod_{l=1}^{\infty}\exp \left(T_l \left( \log \left(1 + \sum_{V \in I(\mathcal{C})} p_l(\mathbf{y}^{(V)})[V]\right) \right) \right) 
\end{eqnarray*}
\noindent
We now use the Baker-Campbell-Hausdorff formula; it provides an expansion for $\log(\exp(A)\exp(B))$ as $\mbox{BCH}(A,B) = A + B + \frac{1}{2}[A, B] + \cdots$, for possibly non-commuting $A,B$ as a linear combination of iterated commutators of $A$ and $B$ (we view the monomials $A$ and $B$ as degenerate commutators). Because $T_l(-)$ respects commutators in the sense of a Lie algebra homomorphism (see Remark \ref{lie_hom_rmk}), we may write:
\begin{eqnarray*}
& &
\exp \left(T_l \left( \log \left(1 + \sum_{U \in I(\mathcal{C})} p_l(\mathbf{x}^{(U)})[U]\right)\right) \right) 
\exp \left(T_l \left( \log \left(1 + \sum_{V \in I(\mathcal{C})} p_l(\mathbf{y}^{(V)})[V]\right) \right) \right) \\
&=&
\exp \left(\mbox{BCH}\left(T_l \left( \log \left(1 + \sum_{U \in I(\mathcal{C})} p_l(\mathbf{x}^{(U)})[U]\right)\right), T_l \left( \log \left(1 + \sum_{V \in I(\mathcal{C})} p_l(\mathbf{y}^{(V)})[V]\right) \right) \right) \right) \\
&=&
\exp \left(T_l\left(\mbox{BCH}\left(\log \left(1 + \sum_{U \in I(\mathcal{C})} p_l(\mathbf{x}^{(U)})[U]\right),  \log \left(1 + \sum_{V \in I(\mathcal{C})} p_l(\mathbf{y}^{(V)})[V]\right) \right) \right)\right) \\
&=&
\exp \left(T_l\left(\log \left(\left(1 + \sum_{U \in I(\mathcal{C})} p_l(\mathbf{x}^{(U)})[U]\right) \left(1 + \sum_{V \in I(\mathcal{C})} p_l(\mathbf{y}^{(V)})[V]\right) \right) \right)\right) \\
\end{eqnarray*}
In the last step we used the fact that $\mbox{BCH}(\log(A), \log(B)) = \log(AB)$ (equivalent to $\mbox{BCH}(A,B) = \log(\exp(A)\exp(B))$). We rewrite our expression to feature only one power-sum symmetric function (albeit with a complicated set of variables). We use the facts that $p_l(\mathbf{x},\mathbf{y}) = p_l(\mathbf{x}) + p_l(\mathbf{y})$ and $p_l(\mathbf{xy}) = p_l(\mathbf{x}) p_l(\mathbf{y})$
\begin{eqnarray*}
& &
\exp \left(T_l\left(\log \left(\left(1 + \sum_{U \in I(\mathcal{C})} p_l(\mathbf{x}^{(U)})[U]\right) \left(1 + \sum_{V \in I(\mathcal{C})} p_l(\mathbf{y}^{(V)})[V]\right) \right) \right)\right) \\
&=&
\exp \left(T_l\left(\log \left(1 + \sum_{U \in I(\mathcal{C})} p_l(\mathbf{x}^{(U)})[U] +  \sum_{V \in I(\mathcal{C})} p_l(\mathbf{y}^{(V)})[V] + \sum_{U, V \in I(\mathcal{C})}p_l(\mathbf{x}^{(U)})p_l(\mathbf{y}^{(V)}) [U][V] \right) \right)\right) \\
&=&
\exp \left(T_l\left(\log \left(1 + \sum_{U \in I(\mathcal{C})} p_l\left( \mathbf{x}^{(U)}, \mathbf{y}^{(U)}, \bigoplus_{V_1, V_2 \in I(\mathcal{C})}
\left(\mathbf{x}^{(V_1)}\mathbf{y}^{(V_2)}\right)^{\oplus N_{V_1, V_2}^{U}} \right) [U]\right)   \right)\right) 
%
\end{eqnarray*}
Here we have used direct sum notation to indicate that $p_l$ should have a collection of symmetric function variables as arguments, and the direct sum in the exponents denotes the multiplicity of each of the sets of variables. For convenience we write $\mathbf{z}^{(U)}$ to denote the family of symmetric function variables $\bigoplus_{V_1, V_2 \in I(\mathcal{C})} \left(\mathbf{x}^{(V_1)}\mathbf{y}^{(V_2)}\right)^{\oplus N_{V_1, V_2}^{U}}$. Note that if the variables $\mathbf{x}$ are indexed as $x_i$, we have
\[
E(t) = \sum_{r \geq 0} e_r(\mathbf{x}) t^r = \prod_{i} (1 + x_i t)
\]
So $E(t)$ (and in particular $E(-1)$) is multiplicative with respect to variable sets.
\[
\left(\sum_{r \geq 0} (-1)^r e_r(\mathbf{x}^{(\mathbf{1})}) \right)
\left(\sum_{s \geq 0} (-1)^s e_s(\mathbf{y}^{(\mathbf{1})}) \right)
=
\left(\sum_{r \geq 0} (-1)^r e_r(\mathbf{x}^{(\mathbf{1})},\mathbf{y}^{(\mathbf{1})}) \right)
\]
Thus our original generating function becomes
\[
\left(\sum_{r \geq 0} (-1)^r e_r(\mathbf{x}^{(\mathbf{1})},\mathbf{y}^{(\mathbf{1})}) \right)
\prod_{l=1}^{\infty}\exp \left(T_l \left( \log \left(1 + \sum_{U} p_l\left(\mathbf{x}^{(U)},\mathbf{y}^{(U)},\mathbf{z}^{(U)}\right)[U]\right) \right) \right) 
\]
This is very close to the generating function of Theorem \ref{bigthm} in variables $\mathbf{x}^{(U)},\mathbf{y}^{(U)},\mathbf{z}^{(U)}$ (only the leading factor is different). Because the leading factor is multiplicative with respect to variable sets, we may write it as
\[
\frac{1}{\sum_{r \geq 0} (-1)^r e_r(\mathbf{z}^{(\mathbf{1})}) }
\sum_{\vv{\lambda} \in \mathcal{P}^\mathcal{C}} \left( \prod_{U \in I(\mathcal{C})} s_{\vv{\lambda}(U)}\left(\mathbf{x}^{(U)},\mathbf{y}^{(U)},\mathbf{z}^{(U)} \right) \right) \otimes X_{\vv{\lambda}}
\]
If the variables $\mathbf{x}^{(U)}$ and $\mathbf{y}^{(V)}$ are indexed as $x_i^{(U)}$ and $y_j^{(V)}$ respectively, the leading term can also be written
\[
\prod_{U,V \in I(\mathcal{C})} \left(\prod_{i,j}\frac{1}{1 - x_i^{(U)}y_j^{(V)}}\right)^{N_{U,V}^{(\mathbf{1})}}
= 
\prod_{U,V \in I(\mathcal{C})} \left( \sum_{\rho \in \mathcal{P}} s_\rho(\mathbf{x}^{(U)}) s_\rho(\mathbf{y}^{(V)})
\right)^{N_{U,V}^{(\mathbf{1})}}
\]
Upon considering the coefficient of $X_{\vv{\lambda}}$ in
\[
\prod_{U,V \in I(\mathcal{C})} \left( \sum_{\rho \in \mathcal{P}} s_\rho(\mathbf{x}^{(U)}) s_\rho(\mathbf{y}^{(V)})
\right)^{N_{U,V}^{(\mathbf{1})}}
\sum_{\vv{\lambda} \in \mathcal{P}^\mathcal{C}} \left( \prod_{U \in I(\mathcal{C})} s_{\vv{\lambda}(U)}\left(\mathbf{x}^{(U)},\mathbf{y}^{(U)},\mathbf{z}^{(U)} \right) \right) \otimes X_{\vv{\lambda}}
\]
we obtain the statement of the theorem.
\end{proof}
\subsection{The Case of $\mathcal{C} = \mbox{Vect}(k)$}
Now we specialise to the case where $\mathcal{C}$ is the category of finite-dimensional vector spaces over $k$. In that case there is only one isomorphism class of simple object $U \in I(\mathcal{C})$, namely $k$ which is idempotent with respect to the tensor structure. As it plays no role, we drop $U=k$ from the notation. Also, $\mathcal{P}^{\mathcal{C}}$ is identified with $\mathcal{P}$. To illustrate how to perform the computation in the statement of Theorem \ref{coefthm}, we prove the following theorem of Littlewood \cite{lw}.
\begin{theorem} \label{identity}
The reduced Kronecker coefficients satisfy the following identity:
\[
\tilde{k}_{\mu, \nu}^{\lambda} =
\sum_{\sigma^{(1)}, \sigma^{(2)}, \sigma^{(3)} \in \mathcal{P}} \sum_{\rho^{(1)}, \rho^{(2)}, \rho^{(3)} \in \mathcal{P}} k_{\sigma^{(2)}, \sigma^{(3)}}^{\sigma^{(1)}}c_{\sigma^{(1)},\rho^{(2)}, \rho^{(3)}}^{\lambda}c_{\rho^{(1)}, \sigma^{(2)},\rho^{(3)}}^{\mu}c_{\rho^{(1)}, \rho^{(2)},\sigma^{(3)}}^{\nu}
\]
Here $k_{\rho^{(1)}, \rho^{(2)}}^{\rho^{(3)}}$ is a Kronecker coefficient, and $c_{\alpha, \beta, \gamma}^{\delta}$ is a (generalised) Littlewood-Richardson coefficient (it is the coefficient of $s_\delta$ in $s_\alpha s_\beta s_\gamma$).
\end{theorem}
\begin{proof}
We consider the case where $\mathcal{C}$ is the category of finite-dimensional vector spaces over $k$ as stated above. Thus $\mathbf{z}^{(k)}$ (in the notation of Theorem \ref{coefthm}) is just $\mathbf{xy}$. Below, all sums are over the set of partitions. The coefficient of $s_\mu(\mathbf{x}) s_\nu(\mathbf{y})$ in the following quantity is the value we wish to calculate.
\begin{eqnarray*}
& &\left( \sum_{\rho^{(1)}}s_{\rho^{(1)}}(\mathbf{x})s_{\rho^{(1)}}(\mathbf{y}) \right) \sum_{\lambda }
s_\lambda(\mathbf{x}, \mathbf{y}, \mathbf{xy})  \\
&=& 
 \sum_{\rho^{(1)}}s_{\rho^{(1)}}(\mathbf{x})s_{\rho_1}(\mathbf{y})  \sum_{\lambda, \rho^{(2)}, \rho^{(3)}}\sum_{\sigma^{(1)}}
c_{\sigma^{(1)} \rho^{(2)}, \rho^{(3)}}^{\lambda}
\left(s_{\rho^{(3)}}(\mathbf{x}) s_{\rho^{(2)}}(\mathbf{y}) s_{\sigma^{(1)}}(\mathbf{xy})\right)  \\
&=&
 \sum_{\rho^{(1)}, \rho^{(2)}, \rho^{(3)}}s_{\rho^{(1)}}(\mathbf{x})s_{\rho^{(1)}}(\mathbf{y})  \sum_\lambda \sum_{ \sigma^{(1)}, \sigma^{(2)}, \sigma^{(3)} }
c_{\sigma^{(1)}, \rho^{(2)}, \rho^{(3)}}^{\lambda} k_{\sigma^{(2)}, \sigma^{(3)}}^{\sigma^{(1)}}
\left(s_{\rho^{(3)}}(\mathbf{x}) s_{\rho^{(2)}}(\mathbf{y}) s_{\sigma^{(3)}}(\mathbf{x})s_{\sigma^{(2)}}(\mathbf{y})\right) \\
&=&
 \sum_{\rho^{(1)}, \rho^{(2)}, \rho^{(3)}}\sum_{ \sigma^{(1)}, \sigma^{(2)}, \sigma^{(3)} }
  k_{\sigma^{(2)}, \sigma^{(3)}}^{\sigma^{(1)}}
\sum_{\lambda}c_{\sigma^{(1)}, \rho^{(2)}, \rho^{(3)}}^{\lambda}
\sum_{\mu }c_{\rho^{(1)}, \sigma^{(2)}, \rho^{(3)}}^{\mu}s_{\mu}({\mathbf{x}})
\sum_{\nu }c_{\rho^{(1)}, \rho^{(2)}, \sigma^{(3)}}^{\nu}s_{\nu}({\mathbf{y}})
\end{eqnarray*}
This completes the proof.
\end{proof}

\noindent
We also point out that Theorem \ref{bigthm} gives a generating function for a known family of symmetric functions, the \emph{irreducible character basis} $\tilde{s}_\lambda$ from \cite{OZ2}. As above we omit $U = k$ entirely from our notation, as well as the tensor product symbols. Theorem \ref{bigthm} becomes
\begin{theorem}
We have the following equality of generating functions.
\[
\sum_{\lambda \in \mathcal{P}} s_\lambda X_\lambda = \left( \sum_{i \geq 0}(-1)^i e_i \right) \prod_{l \geq 1} \left( 1 + p_l\right)^{T_l}
\]
\end{theorem}
\noindent
Let the variables of the symmetric functions present in the above expression be $\mathbf{x}$.
We introduce a new set of symmetric functions in the variables $\mathbf{y}$ such that $\varphi_{\mathbf{1}}(e_i) = \overline{e}_i(\mathbf{1}) + \overline{e}_{i-1}(\mathbf{1})$ is identified with $e_i(\mathbf{y})$. We write $\tilde{s}_\lambda(\mathbf{y})$ for the the symmetric function obtained by writing $X_\lambda$ in terms of the variables $\mathbf{y}$. In accordance with Proposition \ref{TtoP} we have the following equality.
\begin{eqnarray*}
\sum_{\lambda \in \mathcal{P}} s_\lambda(\mathbf{x})\tilde{s}_\lambda(\mathbf{y}) 
&=& \left( \sum_{i \geq 0}(-1)^i e_i(\mathbf{x}) \right) \prod_{l \geq 1} \left( 1 + p_l(\mathbf{x})\right)^{\frac{1}{l}\sum_{d | l} \mu(l/d)p_d(\mathbf{y})} \\
&=&
\left( \sum_{i \geq 0}(-1)^i e_i(\mathbf{x}) \right) \prod_{l \geq 1} 
\sum_{r \geq 0} p_l(\mathbf{x})^r {\frac{1}{l}\sum_{d | l} \mu(l/d)p_d(\mathbf{y}) \choose r}
\end{eqnarray*}
The $\tilde{s}_\lambda$ are polynomials in the elementary symmetric functions such that if the $i$-th elementary symmetric function is replaced with the $i$-th exterior power of the permutation representation of $S_n$, and the multiplication is taken to be the tensor product of $S_n$-representations, then for $n$ sufficiently large, the virtual representation we obtain is the Specht module $\mathcal{S}^{\lambda[n]}$ (this also implies that the characters are obtained by evaluating these symmetric functions at suitable roots of unity, as discussed in \cite{OZ2}). Thus the $\tilde{s}_\lambda$ are fundamental objects in the asymptotic representation theory of symmetric groups. A combinatorial description of them is given in \cite{OZ2}. By comparing the above generating function with their Proposition 11, combined with the description of character polynomials in Example 14 of Section 7 of \cite{macdonald} makes it clear that these are indeed the same symmetric functions.

\section{Generalisation to Tensor Categories}
\noindent
All our results thus far are valid in the setting where $\mathcal{C}$ is a ring category, as per Definition 4.2.3 of \cite{EGNO}, and in particular for any tensor category. That is, $\mathcal{C}$ is an essentially small, locally finite $k$-linear abelian monodial category satisfying two conditions. Firstly, if $\bf{1}$ is the unit object, then $\End_{\mathcal{C}}(\textbf{1}) = k$. Secondly, the product in $\mathcal{C}$ is exact in both arguments and bilinear with respect to direct sums. The essentially small property allows the construction of the Grothendieck group $\Groth$, whilst the artinian property implies that the $\Groth$ is the free abelian group generated by isomorphism classes of simple objects. The exactness of the product in the category implies that it respects the relations of the Grothendieck group and therefore descends to a bilinear distributive multiplication on $\Groth$. Thus, $\Groth$, inherits the structure of a ring. Due to a theorem of Takeuchi, an essentially small $k$-linear Artinian abelian category (in particular, our $\mathcal{C}$) is equivalent to $C$-comod for some coalgebra $C$ over $k$ \cite{takeuchi}.
\newline \newline \noindent
The category of finite-dimensional modules for a bialgebra over $k$ is an example of a ring category (as is the category of finite-dimensional comodules). Generalising this example, the category of finite-dimensional modules over a quasibialgebra is also a ring category.
\newline \newline \noindent
In order to construct wreath product categories, we make use of Deligne's tensor product for categories, which we briefly describe. If $\mathcal{C}_1$ and $\mathcal{C}_2$ are $k$-linear artinian categories, then their tensor product, $\mathcal{C}_1 \boxtimes \mathcal{C}_2$ is another artinian category. It is equipped with a bifunctor $\boxtimes: \mathcal{C}_1 \times \mathcal{C}_2 \to \mathcal{C}_1 \boxtimes \mathcal{C}_2$ satisfying a certain universal property; details can be found in \cite{EGNO}. For our purposes, it suffices to know several properties. Firstly, simple objects in $\mathcal{C}_1 \boxtimes \mathcal{C}_2$ are precisely those of the form $S_1 \boxtimes S_2$ where $S_1$ and $S_2$ are simple objects of $\mathcal{C}_1$ and $\mathcal{C}_2$, respectively. This is a consequence of the fact that if $C_1$ and $C_2$ are coalgebras such that $\mathcal{C}_1$ is equivalent to $C_1-comod$ and $\mathcal{C}_2$ is equivalent to $C_2-comod$, then $\mathcal{C}_1 \boxtimes \mathcal{C}_2$ is equivalent to $(C_1 \otimes C_2)-comod$. Secondly, if $\mathcal{C}_1$ and $\mathcal{C}_2$ are tensor categories, then so is $\mathcal{C}_1 \boxtimes \mathcal{C}_2$, with tensor structure arising from $(X_1 \boxtimes Y_1) \otimes (X_2 \boxtimes Y_2) = (X_1 \otimes X_2) \boxtimes (Y_1 \otimes Y_2)$.
\begin{example}
If $A_1$ and $A_2$ are finite-dimensional $k$-algebras, then $(A_1-mod) \boxtimes (A_2-mod) = (A_1 \otimes_k A_2)-mod$.
\end{example}
\noindent
We may form the $n$-fold Deligne's tensor product of $\mathcal{C}$ which is itself a ring category, which we denote $\mathcal{C}^{\boxtimes n}$. 
\begin{definition}
The equivariantisation of $\mathcal{C}^{\boxtimes n}$ under the natural action of $S_n$ is the \emph{wreath product category} $\Wreath{n} = (\mathcal{C}^{\boxtimes n})^{S_n}$. If $\mathcal{C}$ is a ring category,then $\Wreath{n}$ obtains the structure of a ring category.
\end{definition}
\begin{example}
If $A$ is a finite-dimensional $k$-algebra then $\mathcal{W}_n(A-mod)$ is equivalent to $(A \wr S_n)-mod$, the category of finite-dimensional modules for the wreath product (although $A$ would need some additional structure for $A-mod$ to be a ring category).
\end{example}
\noindent
There is a theory of induction and restriction completely analogous to the theory discussed for finite groups. If a group $G$ acts on objects of $\mathcal{C}$, so does any subgroup, via restriction. Following Section 3.2 of \cite{Mori}, if $\mathcal{D}$ is an additive category, then for any subgroup $H$ of finite index in $G$, we have a forgetful functor $\Res{H}{G}: \mathcal{D}^G \to \mathcal{D}^H$. Additionally there is an induction functor $\Ind{H}{G}: \mathcal{D}^H \to \mathcal{D}^G$ which is both right adjoint and left adjoint to $\Res{H}{G}$. The induction functor may be written as a sum over coset representatives of $H$ in $G$ as follows:
\[
\Ind{H}{G}(M) = \bigoplus_{g \in G/H} gM
\]
In the above formula, the action of $G$ is analogous to that of an induced representation of a finite group.
\newline \newline \noindent
An identical classification of simple objects (i.e. specific objects induced from Young subgroups in the sense described above) of $\Wreath{n}$ holds in greater generality. In \cite{Mori}, this is shown in the context of indecomposable objects of an additive category, but the proof in our setting is analogous.

\bibliographystyle{alpha}
\bibliography{wreath_categories.bib}

\begin{thebibliography}{EGNO15}

\bibitem[Del07]{deligne}
Pierre Deligne.
\newblock La cat{\'e}gorie des repr{\'e}sentations du groupe sym{\'e}trique st,
  lorsque t n'est pas un entier naturel.
\newblock {\em Algebraic groups and homogeneous spaces, Tata Inst. Fund. Res.
  Stud. Math}, pages 209--273, 2007.

\bibitem[Dvi93]{dvir}
Y.~Dvir.
\newblock On the kronecker product of sn characters.
\newblock {\em Journal of Algebra}, 154(1):125 -- 140, 1993.

\bibitem[EGNO15]{EGNO}
P.~Etingof, S.~Gelaki, D.~Nikshych, and V.~Ostrik.
\newblock {\em Tensor Categories}, volume 205 of {\em Mathematical Surveys and
  Monographs}.
\newblock American Mathematical Society, 2015.

\bibitem[Eti14]{etingofRCR1}
Pavel Etingof.
\newblock Representation theory in complex rank, i.
\newblock {\em Transformation Groups}, 19(2):359--381, 2014.

\bibitem[Eti16]{etingofRCR2}
Pavel Etingof.
\newblock Representation theory in complex rank, ii.
\newblock {\em Advances in Mathematics}, 300:473--504, 2016.

\bibitem[Har16]{Nate}
Nate Harman.
\newblock Generators for the representation rings of certain wreath products.
\newblock {\em Journal of Algebra}, 445:125 -- 135, 2016.

\bibitem[Lit58]{lw}
DE~Littlewood.
\newblock Products and plethysms of characters with orthogonal, symplectic and
  symmetric groups.
\newblock {\em Canad. J. Math}, 10:17--32, 1958.

\bibitem[Mac95]{macdonald}
I~.~G. Macdonald.
\newblock {\em Symmetric functions and Hall polynomials}.
\newblock Oxford mathematical monographs. Clarendon Press New York, Oxford,
  second edition, 1995.

\bibitem[Mor12]{Mori}
Masaki Mori.
\newblock On representation categories of wreath products in non-integral rank.
\newblock {\em Advances in Mathematics}, 231(1):1 -- 42, 2012.

\bibitem[Mur38]{murnaghan}
Francis~D Murnaghan.
\newblock The analysis of the kronecker product of irreducible representations
  of the symmetric group.
\newblock {\em American journal of mathematics}, 60(3):761--784, 1938.

\bibitem[OZ16]{OZ2}
Rosa Orellana and Mike Zabrocki.
\newblock Symmetric group characters as symmetric functions.
\newblock {\em arXiv preprint arXiv:1605.06672}, 2016.

\bibitem[Tak77]{takeuchi}
Mitsuhiro Takeuchi.
\newblock Morita theorems for categories of comodules.
\newblock {\em J. Fac. Sci. Univ. Tokyo Sect. IA Math.}, 24(3):629--644, 1977.

\end{thebibliography}

\end{document}